\documentclass[11pt,reqno]{amsart}

\usepackage{amsmath, amsthm, amsopn, amssymb}
\usepackage[nobysame,msc-links]{amsrefs}
\usepackage{enumerate, tikz, etoolbox, intcalc, geometry, caption, subcaption, afterpage}
\usepackage[english]{babel}
\usepackage{mathtools}
\usepackage{color}
\usepackage{MnSymbol}

\newtheorem{remark}{Remark}
\newtheorem{theorem}{Theorem}
\newtheorem{corollary}[theorem]{Corollary}
\newtheorem{lemma}[theorem]{Lemma}
\newtheorem{proposition}[theorem]{Proposition}

\newcommand{\VV}{\mathbb{V}}
\newcommand{\cA}{\mathcal{A}}

\newcommand{\cG}{\mathcal{G}}
\newcommand{\cO}{\mathcal{O}}
\newcommand{\cB}{\mathcal{B}}

\newcommand{\cF}{\mathcal{F}}

\newcommand{\cS}{\mathcal{S}}
\newcommand{\cT}{\mathcal{T}}

\newcommand{\HH}{\mathbb{H}}
\newcommand{\N}{\mathbb{N}}

\newcommand{\Z}{\mathbb{Z}}

\newcommand{\E}{\mathbb{E}}
\newcommand{\bbT}{\mathbb{T}}

\renewcommand{\Pr}{\mathbb{P}}
\newcommand{\1}{\mathbf{1}}
\def\eqd{\,{\buildrel d \over =}\,}
\newcommand{\zero}{\mathbf{0}}

\newcommand{\leD}{\le_{\text{st}}}
\newcommand{\geD}{\ge_{\text{st}}}
\newcommand{\bplus}{\boldsymbol{+}}
\newcommand{\bminus}{\boldsymbol{-}}

\DeclareMathOperator\dist{dist}

\newcommand{\iid}{i.i.d.}
\newcommand{\fiid}{fiid}
\newcommand{\ffiid}{ffiid}

\newcommand{\fvffiid}{fv-ffiid}
\newcommand{\ind}{\mathbf{1}}

\title[Finitary codings for the random-cluster and other infinite-range monotone models]{Finitary codings for the random-cluster model and other infinite-range monotone models}
\date{\today}

\author{Matan Harel}
\address{Matan Harel \hfill\break\indent Northeastern University, Boston, MA, United States of America.}
\email{m.harel@northeastern.edu}
\author{Yinon Spinka}
\address{Yinon Spinka \hfill\break\indent University of British Columbia,
    Department of Mathematics,
 	Vancouver, BC, V6T 1Z2, Canada.}
\email{yinon@math.ubc.ca}
\thanks{Research supported by Israeli Science Foundation grant 861/15, NSERC of Canada, the European Research Council starting grant 678520 (LocalOrder), the Adams Fellowship Program of the Israel Academy of Sciences and Humanities, and the Zuckerman Postdoctoral Scholars Fellowship}

\begin{document}

\begin{abstract}
        A random field $X = (X_v)_{v \in G}$ on a quasi-transitive graph $G$ is a factor of \iid\ if it can be written as $X=\varphi(Y)$ for some \iid\ process $Y= (Y_v)_{v \in G}$ and equivariant map $\varphi$. Such a map, also called a coding, is finitary if, for every vertex $v \in G$, there exists a finite (but random) set $U \subset G$ such that $X_v$ is determined by $\{Y_u\}_{u \in U}$. We construct a coding for the random-cluster model on $G$, and show that the coding is finitary whenever the free and wired measures coincide. This strengthens a result of H{\"a}ggstr{\"o}m--Jonasson--Lyons \cite{haggstrom2002coupling}. We also prove that the coding radius has exponential tails in the subcritical regime. As a corollary, we obtain a similar coding for the subcritical Potts model.

Our methods are probabilistic in nature, and at their heart lies the use of coupling-from-the-past for the Glauber dynamics. These methods apply to any monotone model satisfying mild technical (but natural) requirements. Beyond the random-cluster and Potts models, we describe two further applications -- the loop $O(n)$ model and long-range Ising models. In the case of $G = \Z^d$, we also construct finitary, translation-equivariant codings using a finite-valued \iid\ process~$Y$. To do this, we extend a mixing-time result of Martinelli--Olivieri~\cite{martinelli1994approach} to infinite-range monotone models on quasi-transitive graphs of sub-exponential growth. 
\end{abstract}

\maketitle

\section{Introduction and main results} \label{sec:intro}

Consider an infinite graph $G = (\VV,E)$ and a random field $X = (X_v)_{v \in \VV}$ whose distribution is invariant under all automorphisms of $G$. This paper is concerned with the question of existence of codings (factor maps): is it possible to express $X$ as an automorphism-equivariant function (which we call a coding) an \iid\ process --- i.e. a random field $Y = (Y_v)_{v \in \VV}$ where the $Y_v$'s are independent and identically distributed? The answer to this question depends on the graph $G$ and the random field $X$. The theory of such codings traces back to the seminal work of Ornstein~\cite{ornstein1970bernoulli} and later Keane and Smorodisnky~\cite{keane1979finitary}, who studied the case in which $G = \Z$ and $X$ itself is an \iid\ process. In this case, Keane and Smorodisnky showed that $X$ and $Y$ are finitarily isomorphic -- a stronger condition than the one we require. The study of the one-dimensional problem when $X$ is a more general process remains an active research topic.

In the setting of the $d$-dimensional lattice $\Z^d$, it is very natural to ask whether the Ising model is a factor of an \iid\ process. This model, perhaps the most well-known of the statistical physics models, is infamously trivial on $\Z$, but exhibits a phase transition on $\Z^d$ when $d \geq 2$ -- and hence, it is appropriate to study it on non one-dimensional lattices. In an unpublished work, Ornstein and Weiss~\cite{OW73} (see also~\cite{adams1992folner} for a published version) showed that the (infinite-volume) plus state of the Ising model at any positive temperature is a factor of an \iid\ process. Steif~\cite{steif1991} showed a similar but stronger result for monotone spin systems. In both of these cases, the factor maps may be infinitely dependent, in the sense that determining the value of $X$ at the origin may require knowing the value of infinitely many elements of $Y$. Van den Berg and Steif~\cite{van1999existence} showed that the subcritical Ising model has a {\em finitary} coding. Explicitly, they construct a factor map $\varphi$ from an \iid\ process~$Y$ to any subcritical Ising model such that $\varphi(Y)$ at the origin depends on a finite (but random) number of $Y_v$'s. In fact, their work quantifies the `amount of information' required to determine $\varphi(Y)$ at the origin in two ways. On the one hand, they show that there exists a coding whose coding radius, which controls how far one must look in the $Y$ process, has exponential tails.
On the other hand, they show that there exists a (different) coding which only requires a finite-valued input $Y_v$ at every vertex. The same work shows that no finitary coding can exist for the supercritical Ising model. Recent works constructed finitary codings for Markov random fields with spatial mixing properties~\cite{spinka2018finitarymrf}, or long-range interacting particle systems that satisfy a `high noise' condition~\cite{galves2008perfect}. 

The initial goal of this project was to show that the random-cluster model on $\Z^d$ is a finitary factor of an \iid\ process. Unlike the Ising model, the random-cluster model has infinite-range interactions -- i.e. the state of an edge in the random-cluster model may have a nonvanishing effect on the state of an edge that is arbitrarily far away from it. Although the methods we use yield more general results, the main result of this paper is the construction of a finitary coding for the random-cluster model on an arbitrary quasi-transitive graph when the free and wired measures coincide. In the case of the random-cluster model, the factor constructed in this paper is very similar to the one discussed by H{\"a}ggstr{\"o}m--Jonasson--Lyons~\cite{haggstrom2002coupling}. That paper constructs a factor map for the random-cluster model on a general quasi-transitive graph, but does not study whether it is finitary. The analysis presented herein also provides quantitative control of the coding radius of the factor. In particular, for the subcritical random-cluster model, the coding radius will have exponential tails. We further obtain results for the Potts model on such graphs using the Edwards--Sokal coupling. In the case of the subcritical random-cluster and Potts models on $\Z^d$, we also prove the existence of a finitary coding from a finite-valued \iid\ process.

%
%

The general framework discussed in this paper is that of monotone specifications. Specifications,  a formalization of concepts first introduced in the work of Dobrushin~\cite{Dobrushin1968TheDe} and Landford--Ruelle~\cite{lanford1969observables}, are families of finite-volume measures, indexed by finite subsets and arbitrary configurations (boundary conditions), that satisfy certain consistency relations. They are called monotonic (or attractive) if the measures respect a partial ordering on configurations, in the sense of stochastic domination; this property is a generalization of the FKG property of the random-cluster model, or Griffiths' inequalities for the Ising model. We emphasize that we will not demand that specifications are quasi-local (which is a well-known continuity property), which is frequently assumed elsewhere and violated by the random-cluster  and other infinite-range models. The generality of the framework has many possible applications. We discuss two additional applications: to the critical loop $O(n)$ model on the hexagonal lattice, and to subcritical long-range Ising models. As far as we are aware, this is the first finitary coding result for the loop $O(n)$ model and the only non-perturbative finitary coding result for long-range Ising models (see~\cite{gallo2014attractive} for a result at sufficiently high temperatures).

We end by briefly discussing the algorithmic aspects of our results. There is an extensive literature focused on perfect simulations of infinite-range models~\cite{galves2008perfect,galves2010perfect,de2012developments}; for an example involving the `high noise' regimes of the random-cluster model on $\Z^d$, see~\cite{de2016perfect}. The proofs given in this paper rely on the method of coupling-from-the-past of Propp and Wilson~\cite{propp1996exact}. This technique uses dynamics in order to get a perfect simulation of the stationary distribution of a finite-state Markov chain. 
In our setting, we apply coupling-from-the-past to the single-site Glauber dynamics (in the same spirit as previous works, e.g.,~\cite{van1999existence,haggstrom2000propp}).
As such, there is an interest in controlling not only the spatial dependence of the factor map (i.e. the coding radius), but also the mixing time, which measures the number of steps of the dynamics required to perfectly sample~$X$ at the origin. The celebrated work of Martinelli and Olivieri~\cite{martinelli1994approach} relates spatial and temporal mixing in the context of the finite-range, finite-energy, monotone models on $\Z^d$; as part of this work, we prove a generalization of this result to infinite-range, monotone models on quasi-transitive graphs of sub-exponential growth. With this perspective, it is clear that the existence of space-time finitary factors has algorithmic implications: one may create a perfect sample of $X$ on a finite subset of $\VV$ by applying the space-time finitary factor map to the \iid\ process $Y$ on some random (and possibly much larger) finite subset of $\VV$. Controlling the size of the latter set is tantamount to quantitative control on the coding radius and mixing time.

\subsection{Definitions}

Let $\VV$ be countably infinite and let $\Gamma$ be a group acting on~$\VV$.
The action is \emph{quasi-transitive} if it partitions $\VV$ into finitely many orbits.
Let $(S,\cS)$ and $(T,\cT)$ be two measurable spaces, and let $X=(X_v)_{v \in \VV}$ and $Y=(Y_v)_{v \in \VV}$ be $(S,\cS)$-valued and $(T,\cT)$-valued $\Gamma$-invariant random fields. For the rest of the paper, we will assume that all probability spaces are standard.

A \emph{coding} from $Y$ to $X$ is a measurable function $\varphi \, \colon\,  T^{\VV} \to S^{\VV}$, which satisfies $\varphi(Y) \eqd X$ and is \emph{$\Gamma$-equivariant}, i.e., commutes with the action of every element in~$\Gamma$ on a $\Gamma$-invariant subset of $T^{\VV}$ of full measure. Such a coding is also called a \emph{factor map} or \emph{homomorphism} from $Y$ to $X$; when such a coding exists, we say that $X$ is a \emph{$\Gamma$-factor} of $Y$.

Suppose that $G$ is a locally finite graph on vertex set $\VV$ and that $\Gamma$ acts quasi-transitively on~$\VV$ by automorphisms of $G$. Thus, $G$ is a quasi-transitive graph; heuristically, such a graph has finitely many `different types' of vertices (whereas a transitive graph has exactly one).
We say that a pair of configurations $y,y' \in T^{\VV}$ agree up to distance $r$ around a vertex $v$ if $y_w = y'_w$ for all $v$ with $\dist(v,w) \leq r$, where $\dist(\cdot,\cdot)$ denotes the graph distance. We say that $\varphi$ is determined at distance $r$ around $v$ at a configuration $y$ if $\varphi(y)_v = \varphi(y')_v$ for any $y'$ which agrees with $y$ up to distance $r$ around $v$. The \emph{coding radius} of $\varphi$ at the vertex $v$ and the configuration $y$, which we denote by $R_v(y)$, is the minimal distance that determines $\varphi$ at $v$ and $y$. It may happen that no such~$r$ exists, in which case, $R_v(y)=\infty$. Thus, associated to a coding is a random variable $R_v=R_v(Y)$ which describes the coding radius at $v$; a coding is called {\em finitary} if $R_v(Y)$ is almost surely finite for every $v \in \VV$.\footnote{Technically, the map $R_v(y)$ may not be measurable as defined. One may deal with this by either modifying $\varphi$ on a null set (with respect to $Y$) or by demanding instead that, for almost every $y$, $\varphi(y)_v=\varphi(y')_v$ for almost every $y'$ which agrees with $y$ up to distance $r$ around $v$.}
 
We categorize factor maps as follows: When $X$ is a $\Gamma$-factor of an \iid\ (independent and identically distributed) process, we say it is \emph{$\Gamma$-\fiid}, and when it is a finitary $\Gamma$-factor of an \iid\ process, we say it is \emph{$\Gamma$-\ffiid}. A still stronger notion is $\Gamma$-\fvffiid\, which requires $X$ to be a finitary $\Gamma$-factor of a \emph{finite-valued} \iid\ process (i.e., a finite set $T$). In addition, we can add a quantitative element which indicates how far a coding must look to determine the output at the origin. Explicitly, we say that a coding has \emph{exponential tails} if $\Pr(R_v \ge r) \le Ce^{-cr}$ for some $C,c>0$ and all $r\ge 0$ and $v \in \VV$, and that it has \emph{stretched-exponential tails} if $\Pr(R_v \ge r) \le Ce^{- r^{\nu}}$ for some $C>0$, $0<\nu <1$, and all $r \ge 0$ and $v \in \VV$. For the remains of the paper, when we use the notion \fiid\ (or any variant thereof) without an explicit mention of $\Gamma$,
the group may be taken to be any group acting quasi-transitively on $\VV$ by automorphisms of the graph.

\subsection{The random-cluster model}
We begin with a definition of the random-cluster model; for background on the model and its fundamental properties mentioned below, we direct the reader to the monographs \cite{grimmett2006random,Dum13}.

Let $G' = (V',E')$ be a finite subgraph of $G$, and let $\partial V'$ denote the set of vertices in $V'$ that have a neighbor in $\VV \setminus V'$.
The random-cluster measure in $G'$ with parameters $p\in[0,1]$ and $q>0$ and boundary conditions $i\in\{0,1\}$ is given by
\begin{equation*}
\phi_{G',p,q}^i(\omega)=\frac{p^{o(\omega)}(1-p)^{c(\omega)}q^{k^i(\omega)}}{Z^i(G',p,q)}, \qquad \omega \in \{0,1\}^{E'},
\end{equation*}
where $o(\omega)$ and $c(\omega)$ are the numbers of open and closed edges, i.e., edges $e$ such that $\omega_e=1$ and $\omega_e=0$, respectively, $k^0(\omega)$ is the number of open clusters in $\omega$, $k^1(\omega)$ is the number of open clusters that do not intersect the boundary $\partial V'$, and $Z^i(G',p,q)$ is a normalizing constant, called the partition function, which makes $\phi_{G',p,q}^i$ a probability measure.
We call the measures \emph{free} and \emph{wired} when $i=0$ and $i=1$, respectively.
It is well-known that, when $q \ge 1$, the random-cluster model has the FKG property (a monotonicity property), which implies that $\phi_{G',p,q}^i$ converges weakly to a limiting measure $\phi_{p,q}^i$ as $G'$ increases to $G$.
The two limiting measures are probability measures which are supported on $\{0,1\}^E$ and are invariant under all automorphisms of~$G$. We call $\phi_{p,q}^0$ and $\phi_{p,q}^1$ the free and wired infinite-volume random-cluster measures. Our results concern the coding properties of these two measures.

When $\phi_{p,q}^0=\phi_{p,q}^1$, we may omit the superscript for notational clarity and write $\phi_{p,q}$ for the common measure.
A standard coupling argument shows that, for each $i\in\{0,1\}$, there exists a critical parameter $p_c^i(q) \in [0,1]$ such that
\[
\phi^i_{p,q}[\exists \text{ an infinite cluster}] = \begin{cases} 0 & p < p_c^i(q) \\ 1 & p > p_c^i(q) \end{cases}.
\]
It is also a straightforward consequence of the FKG property that $p_c^0(q) \ge p_c^1(q)$, and that $\phi_{p,q}^0=\phi_{p,q}^1$ whenever $p<p_c^1(q)$. Furthermore, the well-known Burton--Keane argument~\cite{burton1989density} implies that $p_c^0(q) = p_c^1(q)$ on any amenable graph. 

We now state our first result about codings for the random-cluster model.
Recall that our definition of coding is stated for a random field on the vertex set of a graph. Thus, strictly speaking, when considering codings for the random-cluster model on $G$, one should think of the model as being defined on the vertices of the line graph of $G$. In this case, following our convention, the implicit $\Gamma$ may be taken to be any group acting quasi-transitively on $\VV$ by automorphisms of the line graph.

\begin{theorem}\label{thm:fk-coding}
Let $G$ be an infinite quasi-transitive graph and let $p \in [0,1]$ and $q \geq 1$.
\begin{itemize}
\item Both $\phi_{p,q}^0$ and $\phi_{p,q}^1$ are \fiid.
\item If $\phi_{p,q}^0=\phi_{p,q}^1$, then $\phi_{p,q}$ is \ffiid.
\item If $p < p_c^1(q)$, then $\phi_{p,q}$ is \ffiid\ with exponential tails.
\end{itemize}
\end{theorem}

The second and third items are new, while the first item was already known; see the discussion in Section \ref{sec:discussion} for further details.

In the amenable case, we can also prove a partial converse for the second item in Theorem~\ref{thm:fk-coding}.
This converse works on any amenable graph for the wired measure; we require a slightly stronger version of amenability for the free measure.
We say that $G$ is {\em amenable} if there exists a sequence $(F_n)_n$ of non-empty finite subsets of $\VV$ such that $|\partial F_n|/|F_n| \rightarrow 0$ as $n \to \infty$.
We say that $G$ is {\em c-amenable} if there exists a sequence $(F_n,H_n)_n$ of non-empty finite subsets of $\VV$ such that $H_n$ is connected, $\partial F_n \subset H_n$ and $|H_n|/|F_n| \to 0$ as $n\to\infty$. The `c' in c-amenable stands for connected. It is clear that c-amenability implies amenability.

\begin{theorem}\label{cor:Converse}
Let $G$ be an infinite quasi-transitive graph and let $p \in [0,1]$ and $q \ge 1$.
\begin{itemize}
 \item If $G$ is amenable and $\phi_{p,q}^0 \neq \phi_{p,q}^1$, then $\phi_{p,q}^1$ is not \ffiid.
 \item If $G$ is c-amenable and $\phi_{p,q}^0 \neq \phi_{p,q}^1$, then $\phi_{p,q}^0$ is not \ffiid.
\end{itemize}
\end{theorem}

Thus, for c-amenable graphs, Theorem~\ref{thm:fk-coding} and Theorem~\ref{cor:Converse} imply that the wired random-cluster measure is \ffiid\ if and only if the free random-cluster measure is \ffiid, both of which hold if and only if the two measures coincide. In particular, this is the case for $\Z^d$, as it is c-amenable for any $d \ge 2$, and the free and wired are always equal on $\Z$.
We do not know whether the first item of the theorem holds for non-amenable graphs, nor do we know whether the second item holds under the standard amenability assumption. However, the second item does not necessary hold in the non-amenable case. Indeed, on the $d$-regular tree, H{\"a}ggstr{\"o}m~\cite{haggstrom1996random} shows that $p_c^1(q) < p_c^0(q)$ for any $q>2$ and $d \geq 3$, and thus, $\phi_{p,q}^0 \neq \phi_{p,q}^1$ for any $p \in (p_c^1(q), p_c^0(q))$; meanwhile, on any tree and for any $p$ and $q$, $\phi_{p,q}^0$ is exactly Bernoulli percolation of parameter $p/[p + q(1-p)]$, and thus is trivially \ffiid.

The next result is concerned with the existence of codings from a finite-valued \iid\ process for the subcritical random-cluster model on $\Z^d$. In this case, $p_c^0(q) = p_c^1(q)$, so that we may drop the superscript. 
\begin{theorem}\label{thm:fk-fv}
Let $d \geq 2$, $G=\Z^d$ and $\Gamma$ be the translation group of $\Z^d$. Let $q \geq 1$ and $p < p_c(q)$. Then $\phi_{p,q}$ is $\Gamma$-\fvffiid\ with stretched-exponential tails. 
\end{theorem}
We note that the coding we produce above is {\em translation}-equivariant, not \emph{automorphism}-equivariant. The construction we use does not produce a reflection/rotation-equivariant coding, though we believe that such a construction should be possible, and that a similar statement should hold for more general quasi-transitive graphs of sub-exponential growth and their full automorphism group. The proof of Theorem~\ref{thm:fk-fv} relies on a result from~\cite{spinka2018finitaryising} and on a new mixing-time result for a natural single-site dynamics of the subcritical random-cluster model on an arbitrary infinite quasi-transitive graph of sub-exponential growth (see Section~\ref{sec:mixing}).

\subsection{The Potts model}
The random-cluster model is closely related to the Potts model. This model, introduced by Potts \cite{Pot52} following a suggestion of his adviser Domb, has been the subject of intensive study by mathematicians and physicists over the last three decades. For a review of the physics results, see \cite{wu1982potts}; for proofs of the classical rigorous results quoted below, see~\cite{georgii2001random}.

Set $q \geq 2$ to be an integer. The Potts measure on a finite subgraph $G' = (V',E')$ of $G$, at inverse temperature $\beta$ and boundary conditions $i\in\{0,1,\dots,q\}$, is defined by
\begin{equation}\label{eq:Gibbs}
\mu_{G',\beta,q}^i[\sigma]:=\frac{e^{\beta {\bf H}^i_{G'}(\sigma)}}{Z^i_{G',\beta,q}}, \qquad \sigma\in\{1,\dots,q\}^{V'},
\end{equation}
where 
\begin{equation*}
	{\bf H}_{G'}^i(\sigma):=\sum_{\{x,y\}\in E'}\ind[\sigma_x=\sigma_y] + \,\sum_{\{x,y\}\in \partial E'}\ind[\sigma_x=i],
\end{equation*}
and $Z^i_{G',\beta,q}$ is a normalizing constant which makes $\mu_{G',\beta,q}^i$ a probability measure.
Above, $\ind[\cdot]$ denotes the indicator function. Note that when $i=0$, the second sum is zero for all~$\sigma$. One obtains infinite-volume measures $\mu_{\beta,q}^i$ via weak limits, which are known to exist. The case $q=2$ is known as the Ising model.

Like the random-cluster model, the different infinite-volume Potts measures may be highly affected by their boundary conditions. However, if $\beta < \beta_c^\mathrm{w}(q)$, where $\beta_c^{\mathrm{w}}(q) := -\log [1 - p_c^1(q)]$, it is well-known that $\mu_{\beta,q}^0, \mu_{\beta,q}^1, \dots, \mu_{\beta,q}^q$ coincide. In this case, we denote the common measure by $\mu_{\beta,q}$. Using the relation with the random-cluster, we obtain the following results about the coding properties of the Potts model:

\begin{theorem}\label{cor:potts-coding}
Let $G$ be an infinite quasi-transitive graph, $q \ge 2$ be an integer and $\beta \ge 0$.
\begin{itemize}
 \item If  $\mu_{\beta,q}^0, \mu_{\beta,q}^1, \dots, \mu_{\beta,q}^q$ coincide, then $\mu_{\beta,q}$ is \ffiid.
 \item If $\beta<\beta_c^{\mathrm{w}}(q)$, then $\mu_{\beta,q}$ is \ffiid\ with exponential tails.
 \item If $\beta<\beta_c^{\mathrm{w}}(q)$, $G = \Z^d$ and $\Gamma$ is the translation group, then $\mu_{\beta,q}$ is $\Gamma$-\fvffiid\ with stretched-exponential tails.
\end{itemize}
\end{theorem}

In the case of the Ising model $(q=2)$ on $\Z^d$, the results of Theorem \ref{cor:potts-coding} were shown in \cite{van1999existence} (items one and two) and \cite{spinka2018finitaryising} (item three). For $q \neq 2$ on $\Z^d$, partial results were established in \cite{haggstrom2000propp} and \cite{spinka2018finitarymrf}. To the best of our knowledge, the above result is novel in all other cases.
%
%

\subsection{The loop $O(n)$ model}
The loop $O(n)$ model is a model for a random collection of non-intersecting loops on
the hexagonal lattice, which is believed to be in the same
universality class as the spin $O(n)$ model (see~\cite{peled2017lectures} for background on these models).
Let $\Omega$ be a connected finite subset of the hexagonal lattice $\HH$ whose complement is connected, and let $\text{LoopConf}(\Omega)$ be the set of subgraphs where every vertex is of degree 0 or 2, so that the non-trivial connected components form loops. The loop $O(n)$ measure with edge-weight $x>0$ and loop-weight $n>0$ is the probability measure $\nu_{\Omega,n,x}$ given by
\[
\nu_{\Omega,n,x}(\omega) := \frac{x^{o(\omega)} n^{\ell(\omega)} }{Z_{\Omega,n,x}} \cdot \mathbf{1}_{\text{LoopConf}(\Omega)}(\omega),  \qquad \omega \in \{0,1\}^{E(\Omega)},
\]
where $o(\omega)$ is the number of edges in $\omega$, $\ell(\omega)$ is the number of loops in $\omega$, and $Z_{\Omega,n,x}$ is a normalizing constant. The loop $O(n)$ model is conjectured to undergo a phase transition for any $0 \le n\le 2$ when the value of $x$ equals
\begin{equation*}
x_c(n):=\frac{1}{\sqrt{2+\sqrt{2-n}}}.
\end{equation*}

It is shown in~\cite[Theorems 1 and 2]{duminil2017macroscopic} that the loop $O(n)$ model has a unique periodic (i.e. invariant under a finite-index subgroup of the automorphism group of $\HH$) Gibbs measure $\nu_{n,x}$ whenever $n \ge 1$ and $nx^2 \le 1$, and moreover, that this is the unique Gibbs measure whenever $n \in [1,2]$ and $x=x_c(n)$ (we note that a unique periodic measure must be invariant under all automorphisms of $\HH$). In the latter case, we show that this measure has a finitary coding.

\begin{theorem}\label{thm:loop-model}
	Let $n \ge 1$ and $nx^2 \le 1$. Then $\nu_{n,x}$ is \fiid, and it is \ffiid\ when $n\in[1,2]$ and $x=x_c(n)$.
\end{theorem}

It would be interesting to determine whether $\nu_{n,x}$ is always \ffiid\ in the regime $n \ge 1$ and $nx^2 \le 1$.

\subsection{Long-range (ferromagnetic) Ising models}

Let $J=(J_A)_{A \subset \VV, |A| \leq 2}$ be a collection of non-negative numbers called the \emph{coupling constants} satisfying that
\[ \sum_{\substack{A \subset \VV, |A| \leq 2 \\v \in A}} J_A < \infty \qquad\text{for all } v \in \VV .\]
The Ising measure with coupling constants $J$ in a finite volume $V \subset \VV$ with boundary conditions $\tau$ is given by
\[ \mu^\tau_V(\sigma) := \frac{\1_{\{\sigma_{\VV \setminus V}=\tau_{\VV \setminus V}\}}}{Z^\tau_V} \cdot \exp\left[ \displaystyle\sum_{\substack{A \subset \VV, |A| \leq 2 \\A \cap V \neq \emptyset}} J_A \sigma_A \right] ,\qquad \sigma \in \{-1,+1\}^\VV , \]
where $\sigma_A := \prod_{v \in A} \sigma_v$.
We assume that the coupling constants are automorphism-invariant in the sense that $J_{\gamma A} = J_A$ for all $A$ and $\gamma$ in the automorphism group of $G$.

The long-range Ising model has an important monotonicity property, known as Griffiths's inequality~\cite{griffiths1967correlations}, which states that $\E[\sigma_A\sigma_B] \geq \E[\sigma_A] \E[\sigma_B]$ for any finite subsets $A,B \subset \VV$.
It is classical that this property implies monotonicity of the associated specifications, allowing us to define the largest and smallest Gibbs measures $\mu^+$ and $\mu^-$, respectively.

\begin{theorem}\label{thm:ising-coding}
Let $G$ be an infinite quasi-transitive graph and let $J$ be non-negative coupling constants as above. Then $\mu^+$ and $\mu^-$ are \fiid. In addition, if $\mu^+=\mu^-$ then $\mu^+$ is \ffiid.
\end{theorem}

\subsection{Discussion}\label{sec:discussion}

In this section, we discuss the previously known results and their relation to ours. 

We begin with the random-cluster and Potts models, which are our primary interest here. Let us start by explaining the relation between these two models. When $q \ge 2$ is an integer, the two models are closely related via the Edwards--Sokal coupling, which allows to obtain samples of one from the other via a simple procedure (which also introduces additional randomness). In this coupling, to obtain a sample from the Potts measure $\mu^0_{\beta,q}$, one first samples a configuration $\omega$ from the free random-cluster measure $\phi^0_{p,q}$ with $p=1-e^{-\beta}$, and then assigns a single color in $\{1,\dots,q\}$ to all the vertices in each cluster of $\omega$, with the colors of different clusters chosen uniformly and independently. To obtain a sample from $\mu^i_{\beta,q}$ with $i \in \{1,\dots,q\}$, one follows the same procedure with $\omega$ sampled from the wired random-cluster measure $\phi^1_{p,q}$, except that any infinite cluster of $\omega$ (if any such clusters exist) is assigned color $i$. In the other direction, to obtain a sample from $\phi^i_{p,q}$ with $i \in \{0,1\}$, one first samples $\sigma$ from $\mu^i_{\beta,q}$, and then independently opens each edge with probability $p$ if its endpoints have the same color in $\sigma$, and closes it otherwise. This relationship allows to transfer many coding properties from one model to the other. We refer the reader to~\cite[Theorem~4.91]{grimmett2006random} for details.

H{\"a}ggstr{\"o}m--Jonasson--Lyons~\cite{haggstrom2002coupling} studied (non-finitary) coding properties of the random-cluster and Potts models on general graphs. Lemma~4.5 of that paper shows that $\mu_{\beta,q}^i$ with $i \in \{1,\dots,q\}$ is \fiid\ for any integer $q \ge 2$ and any $\beta \ge 0$.
That paper was particularly interested in the closely related notion of Bernoullicity (see Theorem~4.1 there), a classical mixing property from ergodic theory. In fact, on certain amenable graphs with mild geometric conditions -- namely quasi-transitive amenable graphs that satisfy
\[
\VV_{v,r} \setminus \VV_{v,r-1} \not\subset \VV_{u,r} \qquad\text{for any distinct }u,v \in \VV \text{ and infinitely many } r \in \N ,
\]
the notions of Bernoullicity and \fiid\ are equivalent ($\VV_{v,r}$ is the graph-ball of radius $r$ around $v$).
Although it is not explicitly stated there, the first item of Theorem~\ref{thm:fk-coding}, which states that the free and wired random-cluster measures are \fiid, is essentially contained in~\cite{haggstrom2002coupling}. We remark that a slightly stronger version of the above condition appears in Theorem \ref{thm:fv}.

The second and third items of Theorem~\ref{thm:fk-coding} are both new. As far as we know, previous results have been restricted to integer $q$, where they were deduced from results on the Potts model (see~\cite[Remark~5]{spinka2018finitaryising} and \cite[Remark~6]{spinka2018finitarymrf}). In particular, we are unaware of any finitary coding results which work directly on the random-cluster model. Let us also mention that, in the second item of the theorem, the coding radius is controlled by the rate of convergence of the finite-volume free and wired measures (see Theorem~\ref{thm:coding-for-specs}). Thus, the third item of the theorem will follow with the additional knowledge that the phase transition of the random-cluster model is sharp~\cite{duminil2017sharp}.

Theorem~\ref{thm:fk-fv} is concerned with finitary codings in which the \iid\ process is finite-valued. This is a rather natural restriction, as it makes the coding somewhat more useful for simulations. For the Ising model on $\Z^d$, van den Berg and Steif~\cite{van1999existence} constructed such codings, and the second author~\cite{spinka2018finitaryising} showed that such codings exist with stretched-exponential tails for their coding radius. Pushing the methods used to prove Theorem~\ref{thm:fk-coding}, we obtain so-called space-time finitary codings, which, by combining with a general tool from~\cite{spinka2018finitaryising}, allow us to deduce that such codings also exist for the random-cluster model.

Theorem~\ref{cor:Converse} is concerned with showing that the random-cluster measures are not \ffiid\ in certain situations. While this is not the main focus of this paper, the result was given in order to provide a more complete picture. The theorem is an extension of similar results for other models (see \cite[Theorem~2.1]{van1999existence} and \cite[Theorem~1.3]{spinka2018finitarymrf}), which are based on exponential-rate estimates which appear in~\cite{marton1994positive,bosco2010exponential}. Though previous results were for nearest-neighbor models on $\Z^d$, the key ideas apply in greater generality. We do point out however that certain technicalities may arise on other graphs, as is reflected in our need for the c-amenability property in the second item of Theorem~\ref{cor:Converse}. While we were able to verify that certain amenable graphs have this stronger property (e.g., the lamplighter group over $\Z^d$ for any $d \ge 1$), we do not know whether it holds for all infinite, one-ended, amenable, edge-transitive graphs; we note that two-ended graphs such as $\Z$ are not c-amenable.

Let us now turn to the Potts model.
Theorem~\ref{cor:potts-coding} was partially known in the case of $\Z^d$ (all results mentioned here are for $\Gamma$ the group of translations). In the Ising case $q=2$, van den Berg and Steif~\cite{van1999existence} showed the first two items (and a weakening of the third item) and the second author~\cite{spinka2018finitaryising} showed the third item.
In the general case $q \ge 2$, H{\"a}ggstr{\"o}m and Steif~\cite{haggstrom2000propp} showed that $\mu_{\beta,q}$ is \fvffiid\ (and it also follows from the proof that it is \ffiid\ with exponential tails) for $\beta$ small enough, and the second author~\cite[Corollary~1.5]{spinka2018finitarymrf} showed that $\mu_{\beta,q}$ is \ffiid\ with power-law tails for all $\beta<\beta_c(q)$ and that is \ffiid\ when $d=2$, $q \in \{2,3,4\}$ and $\beta=\beta_c(q)$.

There are very few known results regarding the coding properties of the loop $O(n)$ or long-range Ising models for general values of the parameters. Our results for these two models, as well as the random-cluster model, follow from coding results that apply to a general framework of monotone specifications, which will be discussed in Section~\ref{sec:generalframework}. While we developed this methodology with the random-cluster model in mind, it is sufficiently broad to also include the two other models.

Van den Berg and Steif's work \cite{van1999existence}, which we mentioned above in the context of the Ising model, also proves results for monotone Markov random fields. While their methods readily extend to finite-range models, they do not apply to infinite-range models --- i.e., models where the dependence of the conditional distribution at the origin on the boundary conditions is not confined to a bounded box around the origin.
The main technical innovation of our technique is that it allows one to consider such models. For example, the random-cluster model is an infinite-range model. To see this, observe that the conditional distribution of an edge depends not only on its neighboring edges, but rather on the connectivity of its two endpoints, which may require looking arbitrarily far away from the given edge. In fact, this shows that the conditional distribution does not even depend continuously on the boundary condition. 
Our framework requires neither this continuity assumption nor a `high-noise' assumption, both of which are used in works such as \cite{galves2010perfect}. Instead, we rely only on monotonicity and uniqueness, thereby allowing us to obtain non-perturbative results.

\subsection{Organization of the paper}
The next Section \ref{sec:generalframework} introduces the general framework of monotone specifications. In particular, it states Theorem~\ref{thm:coding-for-specs}, Theorem~\ref{thm:fv} and Corollary~\ref{cor:fv-coding}, which are the main technical results of this paper. Section \ref{sec:proofs} proves the theorems introduced above, assuming the theorems of Section \ref{sec:generalframework}. Section \ref{sec:MainCoupling} defines the coupled dynamics that is the basis for constructions of all codings in this paper, and then proves Theorem \ref{thm:coding-for-specs}. Section \ref{sec:mixing} proves a mixing time result, which is then used in the subsequent Section \ref{sec:space-time} to prove Theorem \ref{thm:fv}. Finally, Section \ref{sec:FiniteValued} proves Corollary \ref{cor:fv-coding}.

\subsection{Notation}

We now set up some notation which will be used for the rest of the paper.
Let~$G$ be an infinite locally finite quasi-transitive graph on a countable set $\VV$ (all graphs in this paper satisfy these conditions).
Denote the graph distance in~$G$ by $\dist(\cdot,\cdot)$. For sets $U,V \subset \VV$, we write $\dist(U,V) := \min_{u \in U,v \in V} \dist(u,v)$ and $\dist(u,V) := \dist(\{u\},V)$. We denote $\partial V := \{ u \in \VV : \dist(u,\VV \setminus V)=1 \}$ and $\partial v := \partial \{v\}$. Let
\[ \VV_{v,r} := \{ u \in \VV : \dist(u,v) \le r \} \]
be the ball of radius $r$ around $v$. We also denote
\[ B(r) := \max_{v \in \VV} |\VV_{v,r}| .\]
Recall that $\Gamma$ is a group acting on $\VV$. We extend the action of $\Gamma$ to $A^{\VV}$ (for any set $A$) by
\[ \gamma\omega := (\omega_{\gamma^{-1} v})_{v \in \VV} .\]

Let $\mu$ and $\nu$ be probability measures on a common discrete space $\cA$.
We denote by $\| \mu-\nu \|_{TV}$ the total variation distance between $\mu$ and $\nu$, i.e.,
\[ \| \mu-\nu \|_{TV} := \tfrac12 \sum_{a \in \cA} |\mu(a) - \nu(a)| = \max_{A \subset \cA} |\mu(A) - \nu(A)| .\]
When $\cA$ is partially ordered, we say that $\mu$ is \emph{stochastically dominated} by $\nu$, and write $\mu \leD \nu$, if $\mu(A) \le \nu(A)$ for any increasing event $A$.

\section{Finitary codings for monotone specifications}\label{sec:generalframework}

\subsection{The general framework}
\label{sec:general}

Let $\VV$ be countably infinite, $\Gamma$ be a group acting quasi-transitively on~$\VV$ and $(S,\le)$ be a totally ordered discrete spin space with a maximal element $+$. We extend the order on $S$ to the product partial order on
\[ \Omega:= S^{\VV}, \]
whose maximal element we denote by $\bplus$. Thus, given two elements $\omega,\omega' \in \Omega$,
\[ \omega \le \omega' \quad\iff\quad \omega_v \le \omega'_v \quad\text{for all }v \in \VV .\]
For a finite $V \subset \VV$ and $\tau \in \Omega$, define
\[
\Omega_V^\tau := \{ \omega \in \Omega : \omega_{\VV \setminus V} = \tau_{\VV \setminus V} \}.
\]
Denote
\[ \Omega^+ := \bigcup_{V\subset\VV\text{ finite}} \Omega^{\bplus}_V = \big\{ \omega \in \Omega : \omega\text{ agrees with $\bplus$ outside a finite set} \big\} .\]
An \emph{upwards specification} is a family of measures
\[
\rho = \{ \rho_V^\tau\}_{V \subset \VV\text{ finite},~ \tau \in \Omega^+ } ,
\]
where $\rho_V^\tau$ is a probability measure supported on $\Omega_V^\tau$, that satisfies the consistency relations that, for any finite $U \subset V \subset \VV$ and any $\tau,\tau' \in \Omega^+$,
\[ \rho^\tau_V = \rho^{\tau'}_V \qquad\text{whenever }\tau_{\VV\setminus V} = \tau'_{\VV\setminus V} \]
and
\[ \rho^\tau_V(\,\cdot \mid \Omega^\tau_U) = \rho^\tau_U \qquad\text{whenever }\rho^\tau_V(\Omega^\tau_U)>0 .\]
If we expand this family by defining measures for any $\tau \in \Omega$ and requiring the same consistency relations, we obtain a specification. Upwards specifications are simpler objects than specifications -- for one thing, there are only countably many measures in an upwards specification, whereas a specification requires uncountably many measures.
For any $v \in \VV$, we write $\rho^\tau_v$ as a shorthand for $\rho^\tau_{\{v\}}$.

An upwards specification is \emph{$\Gamma$-invariant} if
\[ \rho_{\gamma V}^{\gamma \tau}(\gamma^{-1} \omega \in \cdot) = \rho_{V}^{\tau} \qquad\text{for any }\gamma \in \Gamma,~ V \subset \VV\text{ finite and }\tau \in \Omega^+. \]
An upwards specification $\rho$ is \emph{irreducible} if, for any finite $V$, the set $\{ \omega \in \Omega^+ : \rho_V^{\bplus}(\omega)>0 \}$ contains~$\bplus$ and is connected in the Hamming graph on $\Omega^+$. An upwards specification is called {\em finite-energy} if there $\inf_{\tau \in \Omega^+, v \in \VV} \rho^\tau_{\{v\}}(\tau) >0$. Intuitively, this condition imposes a uniform lower bound on the cost of changing the configuration at a single vertex. It is straightforward to see that finite-energy is stronger than irreducibility. An upwards specification is \emph{monotonic} if
\[ \rho_{V}^{\tau} \leD \rho_{V}^{\tau'} \qquad\text{for any }V \subset \VV\text{ finite and }\tau,\tau' \in \Omega^+\text{ such that }\tau \le \tau' .\]

When $S$ has a minimal element $-$, we similarly define a notion of a \emph{downwards specification} by replacing $\bplus$ with $\bminus$, the minimal element in $\Omega$, and replacing $\Omega^+$ with $\Omega^-$, the set of configurations which equal $\bminus$ outside a finite set.
When $S$ has both a minimal and maximal element, we may also define a notion of an \emph{upwards-downwards specification}, where $\Omega^+$ is replaced with $\Omega^+ \cup \Omega^-$ above. Such an upwards-downwards specification $\rho$ may be equivalently seen as a pair $(\rho^+,\rho^-)$, where $\rho^+$ is an upwards specification and $\rho^-$ is a downwards specification. In this case, $\Gamma$-invariance of $\rho$ is equivalent to $\Gamma$-invariance of both $\rho^+$ and $\rho^-$, while monotonicity of $\rho$ is equivalent to monotonicity of both $\rho^+$ and $\rho^-$ along with an ordering between $\rho^+$ and $\rho^-$ in the sense that
\[
\rho_V^{\tau} \leD \rho_V^{\tau'} \quad \text{for any }V \subset \VV \text{ finite and } \tau \in \Omega^-, \tau' \in \Omega^+\text{ such that }\tau \le \tau'.
\]
On the other hand, by irreducibility of $\rho$, we mean that both $\rho^+$ and $\rho^-$ are irreducible, without requiring a joint condition.

Let $\rho$ be a monotone upwards specification. By monotonicity, $\rho^{\bplus}_U$ stochastically dominates $\rho^{\bplus}_V$ whenever $U \subset V$. Thus, there exists a weak limit
\[ \mu^+ := \lim_{V \uparrow \VV} \rho^{\bplus}_V .\]
The limit $\mu^+$ is in general a sub-probability measure on $\Omega$ (not necessarily supported on $\Omega^+$), and is $\Gamma$-invariant when $\rho$ is. If $S$ is finite, then $\mu^+$ is a probability measure.
When $\rho$ is a monotone downwards specification, we similarly define $\mu^-$.
In particular, when $\rho$ is an upwards-downwards specification, both $\mu^+$ and $\mu^-$ are well defined.

\subsection{The general results}
We now state the three general results that will be used to prove the main theorems of Section~\ref{sec:intro}.

\begin{theorem}\label{thm:coding-for-specs}
	Let $G$ be an infinite graph on vertex set $\VV$ and let $\Gamma$ be a group acting quasi-transitively on $\VV$ by automorphisms of $G$. Let $S$ be a totally ordered discrete spin space.
	\begin{enumerate}
	\item Suppose that $S$ has a maximal element and let $\rho$ be a monotone $\Gamma$-invariant irreducible upwards specification. If $\mu^+$ is a probability measure, then it is $\Gamma$-\fiid.
	\item Suppose that $S$ is finite and let $\rho$ be a monotone $\Gamma$-invariant irreducible upwards-downwards specification. Then $\mu^+$ is $\Gamma$-\fiid\ with a coding radius that satisfies
	\begin{equation} \label{eq:RadiusByTotalVariation}\Pr(R_v > r) \le (|S|-1) \cdot \big\|\rho^{\bplus}_{\VV_{v,r}}(\sigma_v \in \cdot) - \rho^{\bminus}_{\VV_{v,r}}(\sigma_v \in \cdot)\big\|_{TV} \qquad\text{for all }v \in \VV\text{ and }r \ge 0 .
	\end{equation}
	In particular, if $\mu^+=\mu^-$ then $\mu^+$ is $\Gamma$-\ffiid.
	\end{enumerate}
\end{theorem}

The state spaces for the \iid\ process $Y$ in the above theorem are unrestricted (one may think of $(T,\cT)$ as Lebesgue space on $[0,1]$). In the next section, we wish to control the `amount of temporal information' used by the coding -- heuristically, how many times must the factor map query a (finite-valued) input at any vertex. To this end, we equip the space $(T,\cT)$ with a more explicit structure, namely, we assume that $T=\tilde{T}^\N$, where $\tilde{T}$ is \emph{finite}.
Recall that the coding radius of a coding $\varphi \colon T^{\VV} \to S^{\VV}$ at a vertex $v \in \VV$ and a configuration $y \in T^{\VV}$ is the minimal $r \ge 0$ such that $\varphi(y)_v$ is determined by $(y_w)_{w \in \VV_{v,r}}$.
We analogously define $R^*_v(y)$, the \emph{space-time coding radius} of $\varphi$ at $v$ and $y$, to be the minimal $r \ge 0$ such that $\varphi(y)_v$ is determined by $(y_w(i))_{w \in \VV_{v,r}, 0 \le i \le r}$. We say that such a coding is \emph{space-time finitary} if $R_v^*=R^*_v(Y)$ is almost surely finite for every $v$. In this setting, when $Y$ is said to be an \iid\ process, we mean that $\{Y_v(n)\}_{v \in \VV,n \in \N}$ is a collection of \iid\ random variables supported on the finite set $\tilde{T}$.

We add one final piece of notation before we state the theorem: an upwards-downwards specification $\rho$ is {\em marginally finite} if $\{\rho_v^{\tau}(\sigma_v \in \cdot)\}_{v\in \VV,\,\tau \in \Omega^+\cup\Omega^-}$ is a finite collection of distinct measures.

\begin{theorem}\label{thm:fv}
Let $G$ be an infinite graph on vertex set $\VV$ and let $\Gamma$ be a group acting quasi-transitively on $\VV$ by automorphisms of $G$. Suppose that
\begin{equation}\label{eq:sphere-condition}
\VV_{v,r} \setminus \VV_{v,r-1} \not\subset \VV_{u,r} \qquad\text{for any distinct }u,v \in \VV \text{ and } r\geq 0.
\end{equation}
Let $S$ be a totally ordered finite spin space and $\rho$ be a monotone $\Gamma$-invariant irreducible marginally finite upwards-downwards specification.
\begin{itemize}
 \item If $\mu^+=\mu^-$, then there exists a space-time finitary coding from an \iid\ process $Y$ to $\mu^+$.
 \item Suppose that there exist $C,c>0$ such that
\begin{equation}\label{eq:weak-mixing-cond}
\big\|\rho^{\bplus}_{\VV_{v,r}}(\sigma_v \in \cdot) - \rho^{\bminus}_{\VV_{v,r}}(\sigma_v \in \cdot)\big\|_{TV} \le Ce^{-cr} \qquad\text{for all } v\in \VV \text{ and } r  \ge 0 .
\end{equation}
If $G$ has sub-exponential growth, i.e., $B(r)=\exp(o(r))$, then the tails of the space-time coding radius beats any stretched-exponential, i.e., $\Pr(R^*_v \ge r) \le \exp(-r^{1-o(1)})$. Moreover, if $G$ has growth $B(r)=\exp(o(\frac{r}{\log r}))$, then the space-time coding radius has exponential tails.
\end{itemize}
\end{theorem}

Condition~\eqref{eq:weak-mixing-cond} is commonly referred to as \emph{weak spatial mixing}.
The proof of the second item in Theorem~\ref{thm:fv} relies on controlling the mixing-time of a natural single-site dynamics for specifications satisfying weak spatial mixing (see Section~\ref{sec:mixing}).

In the case in which $\mathbb{Z}^d$ and $\Gamma$ is restricted to translations of the lattice, we can use the setup of~\cite{spinka2018finitaryising} to deduce the existence of \fvffiid\ codings:
\begin{corollary}\label{cor:fv-coding}
Let $G$ be $\mathbb{Z}^d$ or its line graph, $\Gamma$ be the group of translations, $S$ be a totally ordered finite spin space, and $\rho$ be a monotone $\Gamma$-invariant irreducible marginally finite upwards-downwards specification that satisfies~\eqref{eq:weak-mixing-cond}. 
Then $\mu^+$ is $\Gamma$-\fvffiid\ with stretched-exponential tails.
\end{corollary}
We believe that codings from a finite-valued \iid\ process should exist for a much larger class of graphs (perhaps graphs satisfying~\eqref{eq:sphere-condition} and having sub-exponential growth).

\section{Proofs of main results}\label{sec:proofs}

In this section, we prove the theorems stated in Section~\ref{sec:intro}.
All the theorems, with the exception of Theorem~\ref{cor:Converse}, will follow from the general results given in Section~\ref{sec:generalframework}. The proofs of these general results are postponed to the subsequent sections.

\subsection{The random-cluster model}

In this section, we prove the three main theorems about the random-cluster model, namely, Theorem~\ref{thm:fk-coding}, Theorem~\ref{cor:Converse} and Theorem~\ref{thm:fk-fv}. We first place the model in the general framework of Section~\ref{sec:generalframework}.

Consider the random-cluster model with parameters $p \in [0,1]$ and $q \ge 1$ on an infinite quasi-transitive graph $G=(\VV,E)$. While the random-cluster model is most naturally defined on the edges of $G$, our abstract definitions are stated for models defined on the vertex set of a graph, and therefore we view the random-cluster model as a 
`living' on the line graph $\cG$ of~$G$, whose vertex set is $E$. Let $S:=\{0,1\}$ and let $\Gamma$ be a group acting quasi-transitively on $\VV$ by automorphisms of $G$.
The random-cluster model has two natural specifications associated to it -- the free-DLR and wired-DLR specifications -- corresponding to the choice of $i \in \{0,1\}$ in the definition of the model. These are denoted by $\rho^{\text{free}}$ and $\rho^{\text{wired}}$ and defined by
\begin{equation}\label{eq:FK-DLR}
\rho^{\text{\#},\tau}_F(\omega) ~\propto~ p^{o_F(\omega)} (1-p)^{c_F(\omega)} q^{k^{\text{\#}}_F(\omega)} \1_{\Omega^\tau_F}(\omega) , \qquad \omega \in \{0,1\}^E ,
\end{equation}
where $F$ is a finite subset of $E$, $o_F(\omega)$ and $c_F(\omega)$ are the numbers of open and closed edges in $F$, and $k^{\text{\#}}_F(\omega)$ is the number of open clusters that intersect an endpoint of some edge in $F$, with $k^\text{wired}_F(\omega)$ 
counting only finite clusters and $k^\text{free}_F(\omega)$ counting both finite and infinite clusters (we note that the usual notion of `free' and 'wired' boundary conditions correspond to $\rho_F^{\text{free}, \zero}$ and $\rho_F^{\text{wired}, \1}$, where $\zero$ and $\1$ correspond to the all-closed and all-open configurations, respectively).

\begin{remark}\label{remark:1}
A specification $\rho$ gives rise to the notion of a Gibbs measure (also called a DLR state), which is a probability measure $\mu$ such that, for every finite $F \subset E$ and $\mu$-a.e. $\tau \in \{0,1\}^E$, the conditional law of $\omega$ under $\mu$ given that $\omega \in \Omega^\tau_F$ is $\rho^\tau_F$. Thus, in general, the random-cluster model has two notions of Gibbs measures. For amenable quasi-transitive connected graphs, any Gibbs measure (of either of the two specifications) has at most one infinite cluster with probability 1; in this case, there is no distinction between free-DLR and wired-DLR Gibbs measures. In particular, $p_c^0(q) = p_c^1(q)$ for every amenable graph. For more general graphs, the number of infinite clusters may be infinite, in which case, this distinction is essential: for example, $\phi^1_{p,q}$ may not satisfy the free-DLR condition for certain graphs and values of $p$ and $q$, but always satisfies the wired-DLR condition. Similarly, $\phi^0_{p,q}$ may not be a wired-DLR random-cluster measure (see Section 6.4 of \cite{georgii2001random} for further discussion). Also, every Gibbs measure (in either the free or wired DLR sense) stochastically dominates $\phi^0_{p,q}$ and is stochastically dominated by $\phi^1_{p,q}$.
\end{remark}

Let $\rho=(\rho^+,\rho^-)$ be the upwards-downwards specification given by $\rho^+ := (\rho^{\text{wired},\tau}_F)_{F \subset E\text{ finite }, \tau \in \Omega^+}$ and $\rho^- := (\rho^{\text{free},\tau}_F)_{F \subset E\text{ finite }, \tau \in \Omega^-}$.
Whenever $q \ge 1$, the FKG property of the random-cluster model implies that $\rho$ is a monotone specification~(see~\cite[Theorems~3.8 and~2.27]{grimmett2006random}). It is clear from~\eqref{eq:FK-DLR} that $\rho$ is $\Gamma$-invariant. Since, by~\eqref{eq:FK-DLR},
\[ \big\{ \rho^{\text{free},\tau}_e(\sigma_e = s),~ \rho^{\text{wired},\tau}_e(\sigma_e = s) : \tau \in \Omega,~ e \in E,~ s \in \{0,1\} \big\} = \big\{ p,1-p, \tfrac{p}{p+(1-p)q}, \tfrac{(1-p)q}{p+(1-p)q} \big\} ,\]
it is clear that $\rho$ is irreducible and marginally finite.

\begin{proof}[Proof of Theorem~\ref{thm:fk-coding}]
	The free and wired random-cluster measures, $\phi^0_{p,q}$ and $\phi^1_{p,q}$, are precisely the measures $\mu^-$ and $\mu^+$ obtained from the upwards-downwards specification $\rho$.
	Therefore, the first and second items of Theorem~\ref{thm:fk-coding} are immediate consequences of the first and second items of Theorem~\ref{thm:coding-for-specs}, respectively.

	We now turn to the third item of Theorem~\ref{thm:fk-coding}.
    Define $F_{e,r}$ be the set of edges whose distance (taken in the line graph of $G$) from $e$ is at most $r$. For the third item of Theorem~\ref{thm:fk-coding}, it remains only to show that, when $p<p_c^1(q)$, there exist $C,c>0$ such that
    \[ \big\|\rho^{\text{wired},\bplus}_{F_{e,r}}(\sigma_e \in \cdot) - \rho^{\text{free},\bminus}_{F_{e,r}}(\sigma_e \in \cdot)\big\|_{TV} \le Ce^{-cr} \qquad\text{for all } e \in E \text{ and }r \geq 0 .\]
    It is well-known that one may couple samples from $\rho^{\text{wired},\bplus}_{F_{e,r}}$ and $\rho^{\text{free},\bminus}_{F_{e,r}}$ so that they agree on $e$ whenever the endpoints of $e$ are disconnected from the boundary. Thus, if $v$ is an endpoint of $e$,
    \[ \big\|\rho^{\text{wired},\bplus}_{F_{e,r}}(\sigma_e \in \cdot) - \rho^{\text{free},\bminus}_{F_{e,r}}(\sigma_e \in \cdot)\big\|_{TV} \le \phi^1_{\VV_{v,r-1},p,q}(e \leftrightarrow \partial \VV_{v,r-1}) .\]
    The exponential decay of the right-hand side is exactly the content of~\cite[Theorem~1.2]{duminil2017sharp} which establishes the sharpness of the phase transition for the random-cluster model.
\end{proof}

\begin{proof}[Proof of Theorem~\ref{thm:fk-fv}]
In light of the above, Theorem~\ref{thm:fk-fv} is a direct application of Corollary~\ref{cor:fv-coding}.
\end{proof}


\begin{proof}[Proof of Theorem~\ref{cor:Converse}]
We begin with the first item of the theorem, which we proceed to establish by contradiction. Thus, we assume towards a contradiction that $\phi^0_{p,q} \neq \phi^1_{p,q}$ and that $\phi^1_{p,q}$ is $\Gamma$-\ffiid\ for some group $\Gamma$ acting quasi-transitively by automorphisms. Since $G$ is amenable, there exists at most one infinite connected component $\phi^i_{p,q}$-almost surely; thus, $\phi^i_{p,q}$ is both a free-DLR and a wired-DLR Gibbs state (see Remark \ref{remark:1} and \cite[Proposition 6.19]{georgii2001random}). It thus does not matter which specification we work with; we will use the wired-DLR specification for concreteness, but denote it as $\rho$ for notational clarity.

%
	We begin by assuming that $\Gamma$ acts transitively on $E$. We write $\omega$ for a generic random element of $\{0,1\}^E$.
	Since $G$ is amenable, there exists a sequence $F'_n \subset \VV$ of non-empty finite subsets such that $|\partial F'_n|/|F'_n| \to 0$ as $n \to \infty$. Letting $F_n$ denote the set of edges spanned by $F'_n$, and $\partial F_n$ denote the set of edges in $E \setminus F_n$ that share an endpoint with an edge in $F_n$, we have that $|\partial F_n|/|F_n| \to 0$ as $n \to \infty$.
	For $n \ge 1$, denote
	\[ Z_n := \frac{1}{|F_n|} \sum_{e \in F_n} \omega_e .\]
	Denote $a_0 := \phi^0_{p,q}(\omega_e)$ and $a_1 := \phi^1_{p,q}(\omega_e)$ and note that $a_0<a_1$. Since $\phi^1_{p,q}$ is $\Gamma$-\ffiid, it follows that the convergence in the ergodic theorem occurs at an exponential rate for $\phi^1_{p,q}$ (this was shown in~\cite{bosco2010exponential} for the case $G=\Z^d$, and the proof there goes through with no changes for an arbitrary quasi-transitive graph $G$). Hence, denoting $a := \frac12(a_0+a_1)$,
	\[ \phi^1_{p,q}(Z_n \le a) \le Ce^{-2c|F_n|} \qquad\text{for some }C,c>0\text{ and for all }n \ge 1 .\]
	By Markov's inequality and the fact that $\phi^1_{p,q}$ is a wired-DLR Gibbs measure, 
	\[ \phi^1_{p,q}\big(T_n\big) \le Ce^{-c|F_n|}, \qquad\text{where }T_n := \Big\{ \tau \in \{0,1\}^E : \rho^{\tau}_{F_n}(Z_n \le a) \ge e^{-c|F_n|} \Big\} .\]

	Let us show that the all $0$ configuration $\zero$ belongs to $T_n$ for large $n$.
	By monotonicity of the specification $\rho$, $T_n$ is a decreasing set for each $n$, and thus it suffices to show that $T_n$ is non-empty for large $n$. Indeed, by Markov's inequality and the fact that $\phi^0_{p,q}$ is {\em also} a wired-DLR state,
	\begin{align*}
	 1-\phi^0_{p,q}\big(T_n\big)
	  &= \phi^0_{p,q}\Big(\rho^{\omega}_{F_n}(Z_n \le a) < e^{-c|F_n|}\Big) \\
	  &\le \phi^0_{p,q}\Big(\rho^{\omega}_{F_n}(Z_n \le a) \le o(1)\Big) \\
	  &= \phi^0_{p,q}\Big(\rho^{\omega}_{F_n}(Z_n > a) \ge 1-o(1)\Big) \le (1+o(1)) \cdot \phi^0_{p,q}\Big(\rho^{\omega}_{F_n}(Z_n > a)\Big) \\ 
	  & \le  (1+o(1))\cdot \phi^0_{p,q}\big(Z_n > a\big) \le (1+o(1))\tfrac{a_0}{a} .
	\end{align*}
	Since $a_0 < a$, we conclude that $T_n$ is non-empty for large $n$.
	Let $\Omega^\zero_{\partial F_n}$ denote the set of configurations that equal zero on $\partial F_n$.
	Observe that $\rho^{\tau}_{F_n}=\rho^{\zero}_{F_n}$ for all $n$ and $\tau \in \Omega^\zero_{\partial F_n}$, as the free boundary conditions decouple the measure in $F_n$ from the state of $\omega$ on $F_n^c \setminus \partial F_n$. Thus, $\Omega^\zero_{\partial F_n} \subset T_n$ for large $n$, so that
	\[ \phi^1_{p,q}\big(\omega_{\partial F_n} = 0\big) = \phi^1_{p,q}\big(\Omega^\zero_{\partial F_n}\big) \le \phi^1_{p,q}(T_n) \leq  Ce^{-c|F_n|} \qquad\text{for large }n .\]
	On the other hand, by finite energy, we have the lower bound
	\[ \phi^1_{p,q}\big(\omega_{\partial F_n} = 0\big) \ge \left(\frac{p}{p + (1-p)q}\right)^{|\partial F_n|} \qquad\text{for all }n .\]
	Since $\frac{|\partial F_n|}{|F_n|} \to 0$ as $n \to \infty$, we have reached a contradiction. This completes the first item in the case where $\Gamma$ acts transitively on $E$.

To handle the case where $\Gamma$ acts on $E$ quasi-transitively, we let $(O_1,\dots O_k)$ be the (finitely many) orbits of $E$ under $\Gamma$. By possibly extracting a subsequence, we may assume that $r_i = \lim_{n \to \infty} |O_i \cap F_n|/|F_n|$ exists for each $i$. We now define $a_j = \sum_{i=1}^k r_i \cdot \phi^j_{p,q}(\omega_{e_i})$, where $e_i$ is an arbitrary element of $O_i$. By construction, $\phi^j_{p,q}(Z_n)$ converges to $a_j$, and the proof goes through, as above.

	We now turn to the second item of the theorem. The proof given above adapts to this case with some modification which we now explain. First, we choose $F_n$ differently. By definition of c-amenable, there exists a sequence $(F'_n,H'_n)_n$ of non-empty finite subsets of $\VV$ such that $H'_n$ is connected, $\partial F'_n \subset H'_n$ and $|H'_n|/|F'_n| \to 0$ as $n\to\infty$. Let $H_n$ denote the set of edges incident to a vertex in $H'_n$, and let $F_n$ denote the set of edges incident to a vertex in $F'_n$ but not in $H_n$. Then $H_n$ is a connected set of edges, which contains $\partial F_n$ and is disjoint from $F_n$. We replace the occurrences of $\partial F_n$ in the proof with $H_n$. The proof then goes through once we interchange the roles of free and wired, replace the all zero configuration $\zero$ with the all one configuration~$\1$, and replace $Z_n$ by $(1- Z_n)$. The reason for the $H_n$ in this case (and not merely $\partial F_n$) is that, while the $\zero$ configuration acts as a strong `decoupling boundary condition' in the sense that $\rho^{\tau}_V=\rho^{\zero}_V$ whenever $\tau_{\partial V} = 0$, the $\1$ configuration has a weaker decoupling property: in order to force `true wired boundary condition' in the sense that $\rho^{\tau}_V=\rho^{\1}_V$, it is not sufficient to merely have $\tau_{\partial V} = 1$, but rather one needs $\tau_H = 1$ on a set $H$ which connects $\partial V$ from outside of $V$. This is the reason for our assumption of c-amenability. 
\end{proof}

\subsection{The Potts model}
In this section, we prove Theorem~\ref{cor:potts-coding}.
Due to the relation between the Potts and random-cluster models (namely, the Edwards--Sokal coupling), the theorem follows (morally) from the results about the random-cluster (more specifically, the first and second items from Theorem~\ref{thm:fk-coding} and the third item from Theorem~\ref{thm:fk-fv}).
However, before proving Theorem~\ref{cor:potts-coding}, there is a technical issue we must face: Potts configurations belong to $\{1, \dots, q\}^{\VV}$, while random-cluster configurations belong to $\{0,1\}^E$. 
Let $\Gamma$ be a group acting on $\VV$ by automorphisms of $G$. If $\cG$ is the line graph of $G$, then $\Gamma$ can be canonically embedded in the automorphism group of $\cG$. In a slight abuse of notation, we allow $\gamma \in \Gamma$ to act on $E$ through this identification. This allows us to discuss factors from processes on $\VV$ to processes on~$E$.
The lemma below shows that we can produce any \iid\ process on $E$ using an \iid\ process on $\VV$ in a $\Gamma$-equivariant manner.

\begin{lemma}\label{lem:Vertextoedge}
Let $G$ be an infinite quasi-transitive graph and let $\Gamma$ denote its full automorphism group. Then any \iid\ process on $E$ is a finitary $\Gamma$-factor of an \iid\ process on $\VV$ with bounded coding radius.
\end{lemma}
\begin{proof}
Let $X = (X_e)_{e \in E}$ be an \iid\ process, where each $X_e$ takes values in a measurable space $(T,\cT)$. Let $\Delta$ be the maximal degree of $G$. We will show by direct construction that $X$ is a $\Gamma$-factor of the \iid\ process $(Y,Z)=(Y_v,Z_v)_{v \in \VV}$, where $Y_v$ and $Z_v$ are independent, $Y_v=(Y^1_v,\dots,Y^\Delta_v)$ is a collection of $\Delta$ \iid\ random variables having the same distribution as $X_e$, and $Z_v$ is a uniform random variable on $[0,1]$.

Define
\[ \psi \colon (T^\Delta \times [0,1])^\VV \to T^E \]
by
\[ \psi(y^1,\dots,y^\Delta,z)_{\{u,v\}} :=
\begin{cases}
 y^{|\{ w \sim u ~:~ z_u \le z_w \le z_v \} |}_u &\text{if }z_u < z_v\\
 y^{|\{ w \sim v ~:~ z_v \le z_w \le z_u \} |}_v &\text{if }z_u \ge z_v 
\end{cases} .\]
In words, the value associated to an edge $\{u,v\}$ is obtained as follows: $z$ induces an order on the vertices of $G$; the edge $\{u,v\}$ chooses its smaller endpoint with respect to this order -- say, $u$ -- and takes on the value $y^i_u$ for some $i \in \{1, \dots, \Delta\}$. To determine the value of $i$, the set of edges that chose $u$ is ordered according to the value of $z$ at the {\em other} endpoint -- the first edge takes $y^1_u$, the second $y^2_u$, etc.

If $z_u \neq z_v$ for any distinct $u,v \in \VV$, then no value of $y^i_u$ is assigned to more than one edge. Thus, we have $\psi(Y,Z) \eqd X$. Since it is clear that $\psi$ is $\Gamma$-equivariant and has coding radius at most~2, the lemma follows.
\end{proof}

We also require the following simple lemma.
\begin{lemma}\label{lem:ExpTailComposition}
The composition of finitary codings with (stretched-)exponential tails is also a finitary coding with (stretched-)exponential tails.
\end{lemma}
\begin{proof}
	Let $X$, $Y$ and $Z$ be processes on $\VV$.
    Let $\varphi$ be a coding from $Y$ to $X$ and let $\varphi'$ be a coding from $Z$ to $Y$, both having (stretched-)exponential tails. Denote $\tilde{\varphi} = \varphi' \circ \varphi$.
    We denote the coding radii of $\varphi$, $\varphi'$ and $\tilde{\varphi}$ at $v$ by $R_v$, $R'_v$ and $\tilde{R}_v$, respectively.
    
	We first handle the case of exponential tails.    
    Let $C,c>0$ be such that $\Pr(R_v>r) \le Ce^{-cr}$ and $\Pr(R'_v>r) \le Ce^{-c'r}$ for all $r>0$.
    Let $\Delta$ be the maximal degree of $G$ and set $a := c/[2 (c + \log \Delta)]$.
    Fix $v \in \VV$ and define
\[ S_{v,r} := \bigcap_{u \in \VV_{v,ar}} \{ R_u \le (1-a)r \}. \]
By the union bound and the definition of $a$,
\[ \Pr[S_{v,r}^c] \leq B(ar) \cdot Ce^{-c(1-a)r} \leq C\Delta^{ar} e^{-c(1-a)r} \le Ce^{-c r /2} .\]
It straightforward to see that on the event $\{R'_v  \le ar \} \cap S_{v,r}$, we have that $\tilde{R}_v \le r$. Thus,
\[ \Pr[\tilde{R}_v > r] \leq \Pr[S_{r,v}^c] + \Pr[R'_v>ar] \leq Ce^{-cr /2} + Ce^{c' a r} . \]
This completes the proof in the case of exponential tails.

The case of stretched-exponential tails follows in a similar same way, where we let $C>0$ and $0<c,c'<1$ be such that $\Pr(R_v>r) \le Ce^{-r^c}$ and $\Pr(R'_v>r) \le Ce^{-r^{c'}}$, and we set $a := r^{\frac c2-1} / \log \Delta$.
\end{proof}

\begin{proof}[Proof of Theorem \ref{cor:potts-coding}]
Let $\Gamma$ be the full automorphism group of $G$; we set $p = 1 - e^{-\beta}$, and assume that $\mu^0_{\beta,q} = \mu^1_{\beta,q} = \dots = \mu^q_{\beta,q}$. Using the Edwards-Sokal coupling (see~\cite[Theorem~4.91]{grimmett2006random}), this implies that $\phi^0_{p,q} = \phi^1_{p,q}$. By Theorem~\ref{thm:fk-coding}, there exists an \iid\ process $Y$ on the edges, taking values in some measurable space $(T, \mathcal{T})$, and a finitary coding $\varphi: T^{E} \to \{0,1\}^{E}$ from $Y$ to $\phi_{p,q}$. Since $\varphi$ is invariant under any automorphism of the line graph of $G$, we have that $\varphi\circ \gamma = \gamma \circ \varphi$ for any $\gamma \in \Gamma$ (since $\Gamma$ is canonically embedded in the automorphism group of the line graph).

Next, we wish to construct the Edwards--Sokal coupling in a $\Gamma$-equivariant manner. Define
\[ \Psi \colon [0,1]^\VV \times \{1,\dots,q\}^{\VV} \times \{0,1\}^{E} \to \{1,\dots,q\}^{\VV} \]
by
\[ \Psi(z,\sigma,\omega)_v := \sigma_u, \qquad\text{where }u = \text{argmin}\{ z_w :  w \xleftrightarrow{\omega} v \} ,\]
where we recall that $w \xleftrightarrow{\omega} v$ indicates that there exists a path of $\omega$-open edges connecting $w$ and~$v$. Heuristically, $\Psi(z,\sigma,\omega)_v$ outputs the color $\sigma_u$, where $u$ is the vertex in the connected component of $v$ which has the minimal $z$ value. By construction, $\Psi$ is $\Gamma$-equivariant. Let $(Z,\Sigma)$ be an \iid\ process on $\VV$, where $Z_v$ and $\Sigma_v$ are independent and uniform on $[0,1]$ and $\{1,\dots,q\}$, respectively. Then the Edwards--Sokal coupling (see~\cite[Theorem~4.91]{grimmett2006random}) implies that $\Psi(Z,\Sigma,\omega) \eqd \mu_{\beta,q}$ whenever $\omega$ is sampled from $\phi_{p,q}$ independently of $(Z,\Sigma)$. This implies that $\mu_{\beta,q}$ is \ffiid.

For the second item, we assume that $\beta < \beta_c^{\mathrm{w}}(q)$. This implies that $p< p_c(q)$, which implies (by the third item of Theorem \ref{thm:fk-coding}) that the coding radius of $\varphi$ has exponential tails. Suppose now that $Y$ and $(Z,\Sigma)$ are independent. Then the composition $\Psi \circ (\text{id},\text{id},\varphi)$ is a coding from $(Z,\Sigma,Y)$ to $\mu_{\beta,q}$. By Lemma~\ref{lem:Vertextoedge}, this implies that we can create a coding $\Psi'$ from an \iid\ process on $\VV$ to the Potts model $\mu_{\beta,q}$. By Lemma~\ref{lem:ExpTailComposition}, this coding has exponential tails, proving the second item of the theorem.

For the final item, we set $G = \Z^d$ and $\Gamma$ to be the translation group of the lattice. The extra structure here allows us to skip the more complicated constructions above and do things `by hand.' Let $\{e_1,\dots,e_d\}$ denote the the standard basis of $\Z^d$. Any $e \in E$ has a unique representation $e = \{v,v+ e_i\}$, where $v \in \Z^d$ and $1 \le i \le d$.
Define $\tilde{\psi}: (T^d)^{\VV} \to  T^{E}$ by 
\[
\tilde{\psi}(y^1, \dots, y^d)_e := y_v^i, \qquad \text{ where } e = \{v,v+e_i\}.
\]
We also define $\tilde{\Psi}: \{1, \dots, q\}^{\VV} \times \{0,1\}^E$ by 
\[
\tilde{\Psi}(\sigma,\omega)_v := \sigma_u \qquad \text{ where } u = \min \{w : w \xleftrightarrow{\omega} v\}, 
\]
where the minimum over vertices is taken in the lexicographical order on $\mathbb{Z}^d$. Both $\tilde{\psi}$ and $\tilde{\Psi}$ are $\Gamma$-equivariant, as the lexicographical order is translation-invariant. Theorem~\ref{thm:fk-fv} gives us a $\Gamma$-\fvffiid\ coding $\tilde{\varphi}$ for $\phi_{p,q}$ with stretched-exponential tails. Then $\tilde{\Psi} \circ (\text{id},\tilde{\varphi} \circ \tilde{\psi})$ is a $\Gamma$-\fvffiid\ coding for $\mu_{\beta,q}$.
By Lemma~\ref{lem:ExpTailComposition}, the coding radius of this map has stretched-exponential tails, completing the proof.
\end{proof}

\subsection{The loop $O(n)$ model}

In this section, we prove Theorem~\ref{thm:loop-model}.

We define the so-called \emph{spin representation} of the loop $O(n)$ model as follows: set $\VV:=\bbT$ and $S:=\{+,-\}$. For any $\sigma \in S^{\VV}$, define the probability measure~$\mu_V^\tau$ defined by the formula
\begin{equation*}
\label{eq:dilute-Potts-measure}
\mu_V^\tau (\sigma) :=\frac{n^{k(\sigma)} x^{e(\sigma)}}{{\bf Z}_V^\tau} \cdot \1_{\{\sigma_{\VV \setminus V}=\tau_{\VV \setminus V}\}}  ,\,
\end{equation*}
where~$k(\sigma)+1$ is the sum of the number of connected components of pluses and minuses in~$\sigma$ that intersect $V$ or its neighborhood, $e(\sigma) := \sum_{u \sim v} \1_{\sigma_u \neq \sigma_v}$ is the number of edges $\{u,v\}$ that intersect $V$ and have $\sigma_u \neq \sigma_v$, and ${\bf Z}_V^\tau$ is the unique constant making $\mu_V^\tau$ a probability measure. Clearly, both~$k(\sigma)$ and~$e(\sigma)$ depend on~$V$, but we omit it in the notation for brevity. 

The spin representation is related to the original model in the following manner: if $\sigma$ is distributed as $\mu_V^\tau$, then its `domain walls', or the lines that separate $+$ from $-$, are distributed as $\nu_{\Omega,n,x}$ (see \cite[Proposition 3]{duminil2017macroscopic} for more details).

\begin{proof}[Proof of Theorem \ref{thm:loop-model}]
The family of measures $\{\mu_V^{\tau}\}_{V \subset \VV \text{ finite}, \, \tau \in \Omega^+\cup \Omega^-}$ defines an upwards-downwards specification. It is shown in~\cite[Theorem~4]{duminil2017macroscopic} that the spin representation is monotonic whenever $n \ge 1$ and $nx^2 \le 1$. Thus, we can define the infinite-volume limits $\mu^+$ and $\mu^-$, which are the largest and smallest possible measures, respectively. The domain walls of $\mu^+$ and $\mu^-$ are both distributed as $\nu_{n,x}$, the unique periodic Gibbs measure of the loop $O(n)$ model. The operation that maps a spin configuration to its domain walls has a finite coding radius, and therefore Lemma~\ref{lem:Vertextoedge} allows us to transfer coding properties of $\mu^+$ or $\mu^-$ to $\nu_{n,x}$. 
 
The upwards-downwards specification satisfies the finite-energy property, and is therefore irreducible, and is clearly invariant under all automorphisms of $\bbT$. We also note that the automorphism group of the line graph of $\mathbb{H}$ (which is isomorphic to the Kagome lattice) is naturally identified with the automorphism group of $\bbT$. With this in mind, the first part of Theorem \ref{thm:loop-model} follows from the first item of Theorem \ref{thm:coding-for-specs}. For the second part of Theorem \ref{thm:loop-model}, we note that $\mu^+ = \mu^-$ when $n \in [1,2]$ and $x =x_c(n)$, as was shown in \cite{duminil2017macroscopic} (see the last sentence in the paragraph after Theorem 5 in that paper), so that the result follows from the second item of Theorem \ref{thm:coding-for-specs}.
\end{proof}

\subsection{Long-range Ising models}

In this section, we prove Theorem~\ref{thm:ising-coding}. By Theorem \ref{thm:coding-for-specs}, the long-rang Ising model satisfies the desired coding properties if $\mu^+$ and $\mu^-$ are the limit measures of a monotone, irreducible $\Gamma$-invariant upwards-downwards specification. As was mentioned in the discussion above Theorem \ref{thm:ising-coding}, the monotonicity is a classical consequence of Griffith's inequality. Irreducibility follows from the finite-energy property, using the fact that the coupling constants are summable. Finally, $\Gamma$-invariance of the model is a starightforward consequence of the definition of the model and the assumption that the coupling constants are  $\Gamma$-invariant. \qed

\section{Construction of codings via coupling-from-the-past}\label{sec:MainCoupling}

In this section, we introduce a dynamics, stemming from a single-site heat-bath Glauber dynamics, on a general spin system with upwards and downwards specification. The dynamics is defined in any finite volume, given any starting state, and with one of two possible boundary conditions, corresponding to the largest and smallest configurations on the complement. Crucially, the construction ensures that this dynamics couples together all such choices simultaneously. In addition, the dynamics is monotone in the sense that it maintains the partial order on the spin space. The dynamics allows us to use coupling-from-the-past to manufacture an almost-sure limit, which will be the desired coding map. This procedure leads to a proof of Theorem \ref{thm:coding-for-specs}.

\subsection{Overview of the dynamics}
We begin with an informal description of the dynamics, which play a central role in this section.

Consider the sequence of finite graphs $(\VV_{v,r})_{r \in \mathbb{N}}$ and a monotonic irreducible upwards (or upwards-downwards) specification $\rho$. We will define a natural single-site dynamics on each $\VV_{v,r}$, called the $+$ dynamics. A single step of this dynamics started at an arbitrary initial configuration $\omega^{(0)} \in \Omega_{\VV_{v,r}}^{\bplus}$ is defined by applying the following evolution:
\begin{itemize}
\item Order $\VV$ in a $\Gamma$-invariant way and consider the order induced on $\VV_{v,r} = \{v_1,\dots v_m\}$.
\item Obtain $\omega^{(1)}$ from $\omega^{(0)}$ by resampling the value at $v_1$, i.e., $\omega^{(1)}$ is sampled from $\rho^{\omega^{(0)}}_{v_1}$.
\item Repeat inductively, resampling $\omega_{v_k}$ using $\omega^{(k-1)}$, until all sites have been resampled.
\item The final configuration $\omega^{(m)}$ is the new state.
\end{itemize}
When $\rho$ is an upwards-downwards specification, the above can be applied to configurations in $\Omega_{\VV_{v,r}}^{\bminus}$, producing the $-$ dynamics.

In Section~\ref{sec:dynamics}, we construct the $+$ dynamics on $\VV_{v,r}$ so that all initial configurations $\xi$ in $\Omega^+_{\VV_{v,r}}$ are coupled at all times. The irreducibility assumption ensures the dynamics constructed above are ergodic for any fixed $r$. Thus, in the limit as the number of steps of the dynamics tends to infinity, the distribution converges to $\rho^{\bplus}_{\VV_{v,r}}$. Taking $r$ to infinity as well (in a suitable manner) gives convergence in distribution to $\mu^+$. The method of coupling-from-the-past allows us to move from distributional limits to stronger notions of convergence, and thus construct a coding for $\mu^+$, as will be seen in Section~\ref{sec:cftp}. Finally, in Section~\ref{sec:finitary-factor} we consider the $+$ and $-$ dynamics simultaneously (in a properly coupled manner), and deduce that, under the appropriate assumptions, the coding radius $R$ satisfies \eqref{eq:RadiusByTotalVariation}. This allows us to transfer quantitative control on the total variation distance between $\rho^{\bplus}_{\VV_{v,r}}$ and $\rho^{\bminus}_{\VV_{v,r}}$ to quantitative control on the coding radius; in particular, it shows that $\mu^+ = \mu^-$ is sufficient to prove that both measures are \ffiid.

\subsection{The coupled dynamics}
\label{sec:dynamics}

Let $Y=(Y_v)_{v \in \VV}$ be an \iid\ process and suppose that, for each $v \in \VV$,
\[ Y_v=(Y_{v,n})_{n \in \N} \]
is a collection of \iid\ random variables. Further suppose that $(\cA,\pi)$ and $(\cB,\theta)$ are two probability spaces and that, for each $v \in \VV$ and $n \in \N$,
\[ Y_{v,n}=(A_{v,n}, B_{v,n}) \]
are two independent random variables sampled from $\pi$ and $\theta$, respectively.
We denote $A_n:=(A_{v,n})_{v \in \VV}$ and $B_n:=(B_{v,n})_{v \in \VV}$.
The dynamics we construct are functions of $Y$ (specifically, the $n$-th step of the dynamics is a function of $A_n$ and $B_n$), which thus yields a coding $\varphi$ from $Y$ to $\mu^+$.
We now explain how to choose $(\cA,\pi)$, $(\cB,\theta)$ and $\varphi$.
To remain general, we not explicitly define $(\cA,\pi)$, $(\cB,\theta)$ and $\varphi$, but rather let them be arbitrary objects satisfying certain properties required for the proof. This gives us a framework which is sufficiently flexible to prove both Theorem~\ref{thm:coding-for-specs} and Theorem~\ref{thm:fv}. After each definition, we also provide constructions to ensure that the objects we require actually exist. In fact, these will be used for the proof of Theorem~\ref{thm:coding-for-specs}; more delicate versions of these constructions, in which $\cA$ and $\cB$ are finite, will be required for Theorem~\ref{thm:fv} (see Section~\ref{sec:space-time}).

As mentioned above, the dynamics we construct are a coupled version of single-site Glauber dynamics of the given upwards specification. We begin by selecting $(\cA,\pi)$ and a measurable function
\[ F:\Omega^+ \times \VV \times \cA \rightarrow S ,\]
which is used to define a single-site update. Specifically, we require that
\begin{itemize}
\item The random variable $F(\omega,v,\cdot)$ matches the specification at $v$:
\begin{equation}\label{eq:FPlusDef}
\pi\left(F(\omega,v,\cdot ) = s \right) = \rho^{\omega}_v [\sigma_v = s] \qquad\text{for any }\omega \in \Omega^+\text{ and }s \in S .
\end{equation}
\item $F$ is monotonic in $\omega$:
\begin{equation}\label{eq:FPlusMon}
F(\omega,v,a) \leq F(\omega',v,a) \qquad\text{for any }a \in \cA\text{ and }\omega,\omega'\in\Omega^+\text{ such that } \omega \leq \omega',
\end{equation}
\item $F$ is $\Gamma$-invariant:
\begin{equation}\label{eq:FPlusStabInvariant}
F(\omega,v,a) = F(\gamma\omega,\gamma v,a) \qquad\text{for any }a \in \cA,~ \omega \in \Omega^+\text{ and }\gamma \in \Gamma.
\end{equation}
\end{itemize}

At this point, we place no additional restrictions on $\cA$ (in Section~\ref{sec:space-time}, we will need $\cA$ to be finite). This allows to give a simple construction for $F$: set $\cA := [0,1]$, $\pi := \text{Leb}$, the Lebesgue measure on the interval, and, for any $\omega \in \Omega^+$ and $s \in S$, 
\[
a_*(\omega,v,s) := \rho^{\omega}_v [ \sigma_v < s] \qquad\text{and}\qquad a^*(\omega,v,s) := \rho^{\omega}_v [ \sigma_v \leq s] .
\]
It is straightforward to check that, for any $\omega \in \Omega^+$, $s' \in S$ and $a \in (a_*(\omega,v,s'),a^*(\omega,v,s'))$,
\[ \min \{ s \in S : a^*(\omega,v,s) \ge a\} = \max \{ s \in S : a_*(\omega,v,s) \le a\} = s'. \]
Therefore, choosing an arbitrarily $s_0 \in S$, we may now define $F$ by
\[
F(\omega,v,a) := \begin{cases}
 \min \{ s \in S : a^*(\omega,v,s) \ge a\} = \max \{ s \in S : a_*(\omega,v,s) \le a\} &\text{if }a \in \cA' \\
 s_0 &\text{otherwise}
 \end{cases} .
\]
where
\[
\cA' := \bigcap_{\omega \in \Omega^+} \bigcap_{v \in \VV} \bigcup_{s \in S} (a_*(\omega,v,s),a^*(\omega,v,s)) .
\]
This choice of $\cA'$ ensures that $F(\omega,v,a)$ is well defined; we note that $\cA'$ has Lebesgue measure one as its complement is the countable union of Lebesgue measure zero sets.
Using that the upwards specification is monotone and $\Gamma$-invariant, it follows from the definition of $F$ that~\eqref{eq:FPlusMon} and~\eqref{eq:FPlusStabInvariant} hold.
To see that~\eqref{eq:FPlusDef} holds, note that
\[ \pi\left(F(\omega,v,\cdot ) \le s \right) = \text{Leb}(\{ a \in [0,1] : a^*(\omega,v,s) \ge a \}) = a^*(\omega,v,s) = \rho^{\omega}_v [\sigma_v \le s]. \]

For any $v \in \VV$ and $a \in \cA$, we define
\[ F_{v,a} \colon \Omega^+ \to \Omega^+ \]
by
\[ F_{v,a}(\omega)_u := \begin{cases} F(\omega,v,a) &\text{if }u=v\\\omega_u &\text{otherwise} \end{cases} , \qquad \omega \in \Omega^+,~ u \in \VV .\]
Note that~\eqref{eq:FPlusStabInvariant} implies that $F_{v,a}(\omega)=F_{\gamma v,a}(\gamma \omega)$ for $\gamma \in \Gamma$ and $\omega \in \Omega^+$.

We now describe how to choose the updating sites. In our dynamics, the set of updated sites are deterministic; we must, however, be careful as to the {\em order} of the chain of single-site updates which make up a single step of the dynamics. The most straightforward way to order the sites is to associate a uniform $[0,1]$ random variable to each, and use the inherited linear order. This approach is very useful, but is slightly too rigid to allow us to study the coding properties we are interested in -- specifically, this will be an issue when we are looking for codings from a finite-valued \iid\ process. Thus, we give a more abstract definition.

Let $(\cB,\mathfrak{B})$ be a measurable space and let $\mathcal{O} \colon \cB^\VV \times \VV^2 \to \{0,1\}$ be measurable and $\Gamma$-invariant - i.e. 
\[
\mathcal{O}(\gamma \eta, \gamma u, \gamma v) = \mathcal{O}( \eta, u,v) \qquad \text{for all } \gamma \in \Gamma.
\]
We regard $\eta \in \cB^{\VV}$ as inducing via $\cO$ an order $\preceq_\eta$ on $\VV$, where $\mathcal{O} (\eta,u,v)=1$ indicates that $u$ precedes $v$ in this order. Formally, $\preceq_\eta$ is a binary relation on $\VV$, defined by
\begin{equation}\label{eq:order}
 u \preceq_\eta v \qquad\text{if and only if} \qquad \mathcal{O}(\eta,u,v)=1 .
\end{equation}
A general choice of $\mathcal{O}$ and $\eta$ does not result in a linear ordering -- or even a preorder, for that matter! For a probability measure $\theta$ on $\cB$, we say that $\mathcal{O}$ is $\theta$-compatible if $\preceq_\eta$ is almost surely a linear ordering, when $(\eta_v)_{v \in \VV}$ are \iid\ samples from $\theta$. If $\cB = [0,1]$ and $\theta$ is the uniform measure, we can choose $\mathcal{O}(\eta,u,v) = \mathbf{1}_{\eta_u \leq \eta_v}$. This function is clearly $\theta$-compatible, and recovers the simplest ordering described earlier.

Given a finite sequence $q=((v_1,a_1),\dots,(v_k,a_k)) \in (\VV \times \cA)^k$, we denote
\[ F_q := F_{v_1,a_1} \circ \dots \circ F_{v_k,a_k} .\]
Given $\nu \in \cA^{\VV}$, $\eta \in \cB^{\VV}$ such that $\preceq_{\eta}$ is a linear order, a vertex $v \in \VV$ and an integer $r \ge 0$, we define
\[ q(\nu, \eta, v,r) := ((v_1,a_1),\dots,(v_{m},a_{m})) ,\]
where
\[ m=|\VV_{v,r}|, \qquad \VV_{v,r} = \{v_1,\dots,v_{m}\}, \qquad v_1 \preceq_\eta \dots \preceq_\eta v_{m}, \qquad  a_i = \nu_{v_i} .\]
This gives rise to a coupled dynamics on $\Omega$, namely,
\[ F_{q(\nu, \eta, v,r)} \colon \Omega^+ \to \Omega^+ .\]

For any $v \in \VV$ and $r \in \N$, define
\[ Q^{+,v,r} \colon \Omega \to \Omega^{\bplus}_{\VV_{v,r}} \]
to be the natural projection, i.e. 
\[ Q^{+,v,r}(\omega)_u := \begin{cases}
 \omega_u &\text{if } u \in \VV_{v,r}\\
 + &\text{otherwise}
 \end{cases} .\]
This allows us to define the random function from $\Omega$ to $\Omega^+$:
\[ \tilde{f}_n^{+,v,r} := F_{q(A_n,B_n,v,r)} \circ Q^{+,v,r} .\]
The function $\tilde{f}_n^{+,v,r}$ describes the $n$th round of updates in the coupled dynamics. We also define 
\begin{equation}\label{eq:f-def}
 f^{+,v,r}_n := \tilde{f}^{+,v,r}_1 \circ \dots \circ \tilde{f}^{+,v,r}_n .
\end{equation}
There are two important things to note about $f_n^{+,v,r}$. First, the order of composition is reverse from the usual convention. This will prove essential to our construction. For further discussion, see Section~\ref{sec:cftp}. Second, for any $v$, $r$ and $n$, $f^{+,v,r}_n$ is a deterministic function of $Y$.

Having chosen our definitions carefully, we easily obtain the following.
\begin{lemma}\label{lem:ergodicity}
For any $r \ge 0$ and $v \in \VV$,
\[
f^{+,v,r}_n (\bplus) \xrightarrow{(d)} \rho^{\bplus}_{\VV_{v,r}} \qquad \text{as } n\rightarrow \infty.
\] 
\end{lemma}
\begin{proof}
The consistency relations of the upwards specification $\rho^+$ imply that $\rho^{\bplus}_{\VV_{v,r}}$ is stationary with respect to $F_{u,A}$ for any $u \in \VV_{v,r}$, where $A$ is sampled from $\pi$.  Since $Q^{+,v,r}$ is the identity map on $\Omega^{\bplus}_{\VV_{v,r}}$, $\rho^{\bplus}_{\VV_{v,r}}$ is also stationary with respect to $\tilde{f}^{+,v,r}_n$. 

We now consider the countable-state Markov chain 
\[
g_n^{+,v,r} :=\tilde{f}^{+,v,r}_n \circ \dots \circ \tilde{f}^{+,v,r}_1 ,
\]
given by composing the $\tilde{f}_n^{+,v,r}$ in the usual order. The chain is aperiodic (as $\tilde{f}^{+,v,r}_i(\bplus)$ equals $\bplus$ with positive probability) and irreducible (since $\rho^+$ is irreducible). Since there exists a stationary distribution, the Markov chain is ergodic on the states with positive $\rho^{\bplus}_{\VV_{v,r}}$ measure, and thus $g_n^{+,v,r}(\bplus)$ converges in distribution to $\rho^{\bplus}_{\VV_{v,r}}$. Since $g_n^{+,v,r}(\bplus)$ and $f_n^{+,v,r}(\bplus)$ have the same distribution, we are done.
\end{proof}

If we assume that $S$ has a minimal element $-$ and $\rho$ is an upwards-downwards specification, we may extend $F$ to $\Omega^+ \cup \Omega^-$, define a $-$ projection $Q^{-,v,r}$, and thus create $\tilde{f}_n^{-,v,r}$ and $f_n^{-,v,r}$ in order to define the $-$ dynamics. Lemma~\ref{lem:ergodicity} also applies to this dynamics.

\subsection{Monotonicity and existence of factors} \label{sec:cftp}
We can think of $f_n^{+,v,r}$ as a (random) function from $\Omega$ to $\Omega^{\bplus}_{\VV_{v,r}}$ which inherits several monotonicity properties from the upward specification $\rho^+$ (which must be stationary with respect to it).
\begin{lemma}\label{lem:mono}
The function $f_n^{+,v,r}(\omega)$ preserves the order in $\omega$ and is decreasing in~$r$. That is, for any $n,r \geq 0$ and $v \in \VV$, 
\[
f_n^{+,v,r}(\omega) \preceq f_n^{+,v,r}(\omega') \qquad \text{for any }\omega, \omega' \in \Omega\text{ such that } \omega \leq \omega',
\] 
and
\[
f_n^{+,v,r+1}(\omega) \preceq f_n^{+,v,r}(\omega) \qquad \text{for any }\omega \in \Omega. 
\] 
In particular,
\[
f_{n+1}^{+,v,r}(\boldsymbol{+}) \leq f_{n}^{+,v,r}(\boldsymbol{+}) .
\]
\end{lemma}
\begin{proof}
Thanks to \eqref{eq:FPlusMon}, we know that $F_{u,a}(\omega) \leq F_{u,a}(\omega')$ for all $a \in \cA$, $u \in \VV_{v,r}$, and $\omega, \omega' \in \Omega^+$ such that $\omega \leq \omega'$. Furthermore, $Q^{+,v,r}$ also maintains the order on $\Omega$ for fixed $v$ and $r$, and we conclude that $f_n^{+,v,r}$, a composition of monotone functions, is also monotonic.

We now turn to prove the second monotonicity statement. Since $Q^{+,v,r}$ is decreasing in $r$, it suffices to show that, almost surely,
\[ F_{q(A_n,B_n,v,r+1)}(\omega) \le F_{q(A_n,B_n,v,r)}(\omega) \qquad\text{for all }\omega \in \Omega^{\bplus}_{\VV_{v,r}} .\]
Fix $\nu \in \cA^{\VV}$ and $\eta \in \cB^{\VV}$ such that $\preceq_\eta$ is a total order on $\VV$.
Write $\VV_{v,r}=\{v_1,\dots,v_{B(r)}\}$ and $\VV_{v,r+1}=\{u_1,\dots,u_{B(r+1)}\}$, where $v_1 \preceq_\eta \dots \preceq_\eta v_{B(r)}$ and $u_1 \preceq_\eta \dots \preceq_\eta u_{B(r+1)}$.
It is clear that $v_i=u_{i'}$ and $v_j=u_{j'}$ then $i \le j$ if and only if $i' \le j'$.
Therefore, for some functions $G_i \colon \Omega^+ \to \Omega^+$ such that $G_i(\omega)_{\VV_{v,r}}=\omega_{\VV_{v,r}}$ for all $\omega \in \Omega^+$, we have
\begin{align*}
    F_{q(\nu,\eta,v,r)}
     &= \phantom{~G_0 \circ\,}F_{v_1,\nu_{v_1}}\phantom{ \circ G_1~}\circ F_{v_2,\nu_{v_2}} \phantom{ \circ G_2~}\circ \dots \circ F_{v_{B(r)},\nu_{v_{B(r)}}}\phantom{~\circ G_{B(r)}} \\
    F_{q(\nu,\eta,v,r+1)}
     &= G_0 \circ F_{v_1,\nu_{v_1}} \circ G_1 \circ F_{v_2,\nu_{v_2}} \circ G_2 \circ \dots \circ F_{v_{B(r)},\nu_{v_{B(r)}}} \circ G_{B(r)} .
\end{align*}
Observe that, for any $\omega \in \Omega^{\bplus}_{\VV_{v,r}}$, we have $G_i(\omega) \leq \omega$ and $F_{w,\eta_w}(\omega) \in \Omega^{\bplus}_{\VV_{v,r}}$ for any $w \in \VV_{v,r}$. Thus, given $\omega \in \Omega^{\bplus}_{\VV_{v,r}}$, removing every composition with $G_i$ from the above sequence defining $F_{q(\nu,\eta,v,r+1)}(\omega)$ only increases the output, showing that $F_{q(\nu,\eta,v,r+1)}(\omega) \le F_{q(\nu,\eta,v,r)}(\omega)$, as desired.

The final inequality now follows, since $\bplus$ is the maximal element of $\Omega$:
\begin{align*}
f_{n+1}^{+,v,r}(\boldsymbol{+}) & = f_{n}^{+,v,r}(\tilde{f}_{n+1}^{+,v,r} (\boldsymbol{+})) \leq f_{n}^{+,v,r}(\boldsymbol{+}). \qedhere
\end{align*}
\end{proof}

We deduce a simple but crucial consequence of the above lemma and the definition of $f_n^{+,v,r}$:
\begin{corollary}\label{cor:ASconvergence}
The random field
\[
\sigma^{+,v,r} :=\lim_{n \rightarrow \infty } f_{n}^{+,v,r}(\boldsymbol{+}) 
\]
is defined almost-surely and has the distribution $\rho^{\bplus}_{\VV_{v,r}}$. Furthermore, if $\mu^+$ is a probability measure,
\[
\sigma^+:= \lim_{r\rightarrow \infty} \sigma^{+,v,r} = \lim_{n,r \rightarrow\infty } f_{n}^{+,v,r}(\boldsymbol{+})
\]
is also a well-defined random field, independent of $v$, with distribution $\mu^+$.
\end{corollary}
\begin{proof}
The sequence $\{f_n^{+,v,r}(\boldsymbol{+})\}_n$ is decreasing, and must have an almost-sure limit, taking value in $\overline{S}^{\VV}$ for some possibly larger $\overline{S} \supset S$ (if $S$ is not finite, and thus not compact in the discrete topology, we cannot be sure that the limit is supported on $\Omega$ {\em a priori}). However, by Lemma \ref{lem:ergodicity}, the distribution of the limiting random variable is known to be $\rho^{\bplus}_{\VV_{v,r}}$, which is supported on $\Omega^+$, giving the desired result.

By Lemma \ref{lem:mono}, $\sigma^{+,v,r}$ is decreasing in $r$, meaning it, too, has an almost-sure limit in a possibly larger space as $r$ goes to infinity. Thanks to the monotonicity of the specifications, the resulting limit is independent of $v$. Since $\mu^+$ is defined by exhaustion, $\sigma^+$ must have distribution $\mu^+$, meaning $\sigma^+$ is almost-surely supported in $\Omega$. To complete the proof, we note that the array $\{f_n^{+,v,r}(\bplus)\}_{n,r}$ is monotonically decreasing pointwise in both $n$ and $r$, and thus extracting any diagonal sequence maintains the almost-sure convergence properties above.
\end{proof}

\begin{remark}[Coupling-from-the-past]\label{remark:2}
Let us momentarily reconsider the standard, "forward" dynamics given by $g_{n}^{+,v,r}:= \tilde{f}_n^{+,v,r} \circ \dots \circ \tilde{f}_1^{+,v,r}$. In distribution, the random variables $f_n^{+,v,r}(\boldsymbol{+})$ and $g_n^{+,v,r}(\boldsymbol{+})$ are identical for any fixed $n$. However, $g_n^{+,v,r}(\boldsymbol{+})$ cannot be monotonic in $n$ - if our configuration $\boldsymbol{+}$ evolved to be some different configuration $\omega$ at time $n$, there is no reason to believe that the next step in the dynamic is smaller than $\omega$! In fact, $g_n^{+,v,r}(\boldsymbol{+})$ does not have an almost-sure limit, as it continues changing after every application of the $g_n$.

On the other hand, $f_n^{+,v,r}(\omega)$ is defined from the past, and can be thought of as evaluating the forward dynamics at time 0, with `initial' conditions of $\omega$ at time $-n$. Since $f_n^{+,v,r}$ is a function from $\Omega$ to $\Omega^{+,v,r}$, we can sample $f_{n+1}^{+,v,r}$ given $f_n^{+,v,r}$ by taking $\omega$ to the (random) configuration $\tilde{f}_{n+1}^{+,v,r}(\omega)$, and then mapping it to the (deterministic, given the conditioning) configuration assigned to it by $f_n^{+,v,r}$. This concatenation construction is crucially important for the final inequality of Lemma \ref{lem:mono}, and is the conceptual justification for the existence of almost-sure limits in coupling-from-the-past. 
\end{remark}

\begin{proof}[Proof of Theorem~\ref{thm:coding-for-specs}, item 1]
The existence of an \iid\ coding follows by explicit construction: for any $v \in \VV$, $\sigma^+_v = \lim_{n,r \rightarrow \infty}f_{n}^{+,v,r}(\boldsymbol{+})_v$. Since $\mu^+$ is a probability measure by assumption, Corollary~\ref{cor:ASconvergence} shows that $\sigma^+$ is distributed as $\mu^+$; the $\Gamma$-invariance of the specifications implies that $\sigma^+$ is a deterministic and $\Gamma$-equivariant function of~$Y$.
\end{proof}

\subsection{Finitary factors via quantitative bounds on coding radius}
\label{sec:finitary-factor}

For this section, we assume that $S$ is a finite spin space and that $\rho$ is a monotone $\Gamma$-invariant irreducible upwards-downwards specification. In this case, $S$ has both a maximal and minimal element, and both $\mu^+$ and $\mu^-$ are probability measures.

As mentioned above, the construction in Section~\ref{sec:dynamics} extends to upwards-downwards specifications. Observe that $f_n^{-,v,r}$ enjoys similar properties as $f_n^{+,v,r}$, with the notable difference that $f_n^{-,v,r}(\omega)$ is \emph{increasing} in $r$ (it still preserves the order in $\omega$). As in Corollary~\ref{cor:ASconvergence}, $\sigma^{-,v,r}$ and $\sigma^-$ are defined almost surely and are distributed as $\rho^{\bminus}_{\VV_{v,r}}$ and $\mu^-$, respectively.

We stress that the $+$ and $-$ dynamics are coupled as they are both defined through the same process $Y$. In particular, almost surely,
\[
f_n^{-,v,r}(\omega) \le f_n^{+,v,r}(\omega') \quad \text{ for any } v \in \VV,~ r \in \mathbb{N}\text{ and }\omega, \omega' \in \Omega\text{ such that } \omega \leq \omega' .
\] 
This implies that, almost surely,
\[
\sigma^{-,v,r} \leq \sigma^{-} \leq \sigma^{+} \leq  \sigma^{+,v,r} \quad \text{ for any } v \in \VV.
\]

For finite spin spaces, we have the following lemma which relates the probability of disagreement under monotone couplings to total-variation bounds:
\begin{lemma}\label{lem:TV}
Let $X$ and $Y$ be random variables taking value in a totally ordered, finite spin space~$S$. If $\mathbb{P}[X \leq Y] =1$, then 
\[
\mathbb{P}[X \neq Y] \leq (|S|-1) \cdot \|X-Y \|_{TV}.  
\]
\end{lemma}
\begin{proof}
Identifying $S$ with the set $\{0, \dots, |S|-1\}$, we see that, by Markov's inequality, 
\begin{align*} 
\Pr[X \neq Y] = \Pr[Y - X \geq 1] \leq \E[Y - X] = \E_{\text{op}}[Y-X] \leq (|S|-1) \Pr_{\text{op}}[X \neq Y], 
\end{align*}
where $\Pr_{\text{op}}$ is an optimal coupling between $X$ and $Y$. In an optimal coupling, the probability of two variables not matching is exactly the total variation distance, as required.
\end{proof}

\begin{proof}[Proof of Theorem~\ref{thm:coding-for-specs}, item 2]
We let $(\cA,\pi)$ and $F$ be as above. We let $(\cB,\theta)$ be the Lebesgue measure space on $[0,1]$ and set $\mathcal{O}(\eta,u,v) := \mathbf{1}_{\eta_u \leq \eta_v}$. With this choice, it is clear that $\sigma^{+,v,r}$ is measurable with respect to $Y_{\VV_{v,r}}$. Recall that $\sigma^+$ describes a coding from $Y$ to $\mu^+$. Our goal is then to bound its coding radius $R_v$.

Define the random variable 
\[
\tilde{R}_v := \min \big\{ r \ge 0 : \sigma^{+,v,r}_v = \sigma^{-,v,r}_v \big\},
\] 
where we set $\tilde{R} = \infty$ if the two spins do not agree for any $r$. Thus, $\sigma^+_v = \sigma^{+,v,r}_v$ for any $r \ge \tilde{R}_v$. Since $\sigma^{+,v,r}$ is independent of $(Y_u)_{u \not \in \VV_{v,r}}$ and $\tilde{R}_v$ is a stopping time with respect to the filtration of $(Y_{\VV_{v,r}})_r$, it is clear that $R_v \le \tilde{R}_v$.

By Lemma~\ref{lem:TV}, 
\[
\mathbb{P}[\tilde{R}_v >r]  = \mathbb{P}[\sigma^{+,v,r}_v \neq \sigma^{-,v ,r}_v] \leq (|S|-1) \cdot \big\|\rho^{\bplus}_{\VV_{v,r}}(\sigma_v \in \cdot) - \rho^{\bminus}_{\VV_{v,r}}(\sigma_v \in \cdot)\big\|_{TV},
\]
as required. In particular, if $\mu^+ = \mu^-$, it is clear the total variation distance must vanish as $r \to \infty$, so that $R_v$ is almost-surely finite.
\end{proof}

\section{Weak spatial mixing implies exponential mixing in time}
\label{sec:mixing}

In this section, we prove bounds on the mixing-time of the dynamics considered in Section~\ref{sec:MainCoupling}. Throughout this section, we will assume that $S$ is finite and that we are given a monotone $\Gamma$-invariant irreducible upwards-downwards specification $\rho$. In particular, we use the coupled~$+$ and~$-$ dynamics as in Section~\ref{sec:finitary-factor}.

A monotone $\Gamma$-invariant (upwards-downwards) specification $\rho$ is said to satisfy \emph{weak spatial mixing} with rate $c>0$ if~\eqref{eq:weak-mixing-cond} holds for some $C$.
Martinelli and Olivieri \cite{martinelli1994approach} show that when $G=\Z^d$ and $\rho$ is a monotone $\Gamma$-invariant specification satisfying a \emph{finite-range assumption} and a \emph{finite-energy assumption}, weak spatial mixing implies that the mixing-time of the single-site Glauber dynamics (as considered in Section~\ref{sec:MainCoupling}) has exponential tails (their setting is a continuous-time dynamics on $\Z^d$, but the proof easily adapts to our discrete-time dynamics on $\Z^d$). In our notation, this means that the total-variation distance between $\lim_{r \to \infty} f_n^{+,v,r}(\bplus)_v$ and $\lim_{r \to \infty} f_n^{-,v,r}(\bminus)_v$ is exponentially small in~$n$. Using the order constructed in Section~\ref{sec:dynamics}, we can see that the finite range assumption implies that $f_n^{+,v,n}(\bplus)_v$ and $f_n^{-,v,n}(\bminus)_v$ are exponentially close, as both $f_n^{+,v,r}(\bplus)$ and $f_n^{-,v,r}(\bminus)$ depend only on $\{A_{u,i},B_{u,i}\}_{u \in \VV_{v,Cr}, i \le n}$ for some constant $C>0$ depending on the range of the specification.

We extend the result of Martinelli--Olivieri in a number of directions. First, we allow an arbitrary quasi-transitive graph $G$ of sub-exponential growth (though we require a slightly stronger quantitative bound on the rate of growth for the full conclusion). Second, we drop the finite-range and finite-energy assumptions, requiring only an irreducibility assumption. Third, we work with a monotone upwards-downwards specification, instead of a (usual) monotone specification. Lastly, we keep track not only of the amount of time required until mixing, but also the amount of \emph{space} (in the graph $G$) required.

\begin{theorem}\label{thm:mixing}
Let $G$ be an infinite graph and $\Gamma$ be a group acting quasi-transitively on $\VV$ by automorphisms of $G$.
Let $S$ be a totally ordered finite spin space and $\rho$ be a monotone $\Gamma$-invariant irreducible upwards-downwards specification that satisfies weak spatial mixing with rate $c>0$.
If $G$ has sub-exponential growth, i.e., $B(r)=e^{o(r)}$ as $r \to \infty$, then
	\[
	\max_{v \in \VV} \big\|f_n^{+,v,n}(\bplus)_v - f_n^{-,v,n}(\bminus)_v\big\|_{TV} \le e^{-n^{1-o(1)}} \qquad\text{as }n \to \infty .
	\]
	Moreover, if there exists $\beta<c\log 2$ such that $B(r) \le e^{\frac{\beta r}{\log r}}$ for large $r$, then there exists $c'>0$ such that
	\[
	\max_{v \in \VV} \big\|f_n^{+,v,n}(\bplus)_v - f_n^{-,v,n}(\bminus)_v\big\|_{TV} \le e^{-c'n} \qquad\text{for large }n.
	\]
\end{theorem}

	For the proof, we require the following calculus lemma, whose proof we postpone to the end of the section.
	
	\begin{lemma}\label{lem:function-decay}
		Let $\psi \colon \N \to (0,\infty)$ be monotone decreasing to zero and let $b \colon \N \to [0,\infty)$ be sub-linear. Suppose that, for some $C,c>0$,
		\[ \psi(2n) \le e^{b(s)} \psi(n)^2 + C e^{-cs} \qquad\text{for all }n \ge s \geq 1.\]
		Then $\psi(n)$ decays faster than any stretched-exponential, i.e., $\psi(n) \le \exp(-n^{1-o(1)})$. In addition, if $b(n) \le \frac{\beta n}{\log n}$ for some $\beta<c\log 2$ and all sufficiently large~$n$, then $\psi(n)$ decays exponentially fast.
	\end{lemma}

\begin{proof}[Proof of Theorem~\ref{thm:mixing}]
	For $n,r \ge 0$, define
	\[ \phi(n,r) := \max_{v \in \VV} \Pr[f_n^{+,v,r}(\bplus)_v \neq f_n^{-,v,r}(\bminus)_v]. \]
	Throughout the proof, we repeatedly use Lemma~\ref{lem:mono} without explicit mention; in particular, we use the fact that $\phi(n,r)$ is decreasing in both $n$ and $r$, as easily follows.

	It suffices to show that $\psi(n):=\phi(n,n)$ has the desired decay rate.
	Note that the weak spatial mixing assumption implies that $\mu^+=\mu^-$, which, together with irreducibility, implies that $\psi(n) \to 0$ as $n \to \infty$.
	Thus, the theorem will follow from Lemma~\ref{lem:function-decay} once we establish the following inequality:
	\[ \psi(2n) \leq \phi(2n,n + s) \le 2B(s) \psi(n)^2 + 3C|S|e^{-cs}  \qquad\text{for all } n \geq s \ge 0.\]
	In fact, we show the slightly stronger inequality: 
	\begin{equation}\label{eq:main-recursion}
	\phi(n+m,r+s) \le 2B(s) \phi(n,r) \phi(m,r+s) + 2C|S|e^{-cs} + C|S|e^{-cr}  \;\text{for all }n,m,s \ge 0 \text{ and } r \geq s .
	\end{equation}	
	The earlier inequality follows by setting $n = m = r$.

	Recall that $f_n^{\pm,v,r}$ is measurable with respect to the \iid\ process~$Y$.
	Let $\xi_v^+ \sim \rho^{\bplus}_{\VV_{v,r}}$ and $\xi_v^- \sim \rho^{\bminus}_{\VV_{v,r}}$ be a random variables, independent of $Y$, satisfying that $\xi_v^+ \ge \xi_v^-$ almost surely.
	Note that $f_n^{+,v,r}(\xi_v^+) \sim \rho^{\bplus}_{\VV_{v,r}}$ and $f_n^{-,v,r}(\xi_v^-) \sim \rho^{\bminus}_{\VV_{v,r}}$.
	By comparing $f_n^{+,v,r}(\bplus)_v$ to $f_n^{+,v,r}(\xi_v^+)_v$ and $f_n^{-,v,r}(\bminus)_v$ to $f_n^{-,v,r}(\xi_v^-)_v$ and using Lemma~\ref{lem:TV}, we see that
	\begin{equation}\label{eq:phi-plus-minus}
	\phi(n,r) \le \phi^+(n,r) + \phi^-(n,r) + |S| \cdot \max_{v \in \VV} \big\|\rho^{\bplus}_{\VV_{v,r}}(\sigma_v \in \cdot) - \rho^{\bminus}_{\VV_{v,r}}(\sigma_v \in \cdot)\big\|_{TV} ,
	\end{equation}
	where
	\begin{align*} \phi^+(n,r) & := \max_{v \in \VV} \Pr[f_n^{+,v,r}(\bplus)_v \neq f_n^{+,v,r}(\xi_v^+)_v], \\  \phi^-(n,r) & := \max_{v \in \VV} \Pr[f_n^{-,v,r}(\bminus)_v \neq f_n^{-,v,r}(\xi_v^-)_v] .\end{align*}
	Next, we now show that
	\begin{equation}\label{eq:recursion}
	 \phi^{\pm}(n+m,r+s) \le B(s) \phi(n,r) \phi(m,r+s) + |S| \cdot \max_{v \in \VV}  \big\|\rho^{\bplus}_{\VV_{v,s}}(\sigma_v \in \cdot) - \rho^{\bminus}_{\VV_{v,s}}(\sigma_v \in \cdot)\big\|_{TV} .
	 \end{equation}
	This statement will give~\eqref{eq:main-recursion} thanks to~\eqref{eq:weak-mixing-cond} and~\eqref{eq:phi-plus-minus}.	 
	We show~\eqref{eq:recursion} only for $\phi^+$ as the proof for $\phi^-$ is similar. Recall from \eqref{eq:f-def} that $f_{n+m}^{+,v,r+s}$ is the composition of $n+m$ independent copies of $\tilde{f}_1^{+,v,r+s}$. Letting $h_n^{+,v,r+s}$ be identical in distribution to $f_n^{+,v,r+s}$ and independent of $Y$ and $\xi_v^+$, we see that $f_m^{+,v,r+s} \circ h_n^{+,v,r+s}$ has the same distribution as $f_{n+m}^{+,v,r+s}$. Therefore,
	\begin{align*}
	   \phi^+(n+m,r+s)  = \max_{v \in \VV} \Pr\left[ f_m^{+,v,r+s}(h_n^{+,v,r+s}(\bplus))_v \neq f_m^{+,v,r+s}(h_n^{+,v,r+s}(\xi_v^+))_v \right], 
	\end{align*}
Now, letting $E_{v,u}$ denote the event that $h_n^{+,v,r+s}(\bplus)_u=h_n^{+,v,r+s}(\xi_v^+)_u$ and letting $E_v := \bigcap_{u \in \VV_{v,s}} E_{v,u}$,
	\begin{align*}
	 \phi^{\pm}(n+m,r+s)
	  &\le \max_{v \in \VV} \Pr(E_v^c) \cdot \Pr\left[ f_m^{+,v,r+s}(\bplus)_v \neq f_m^{+,v,r+s}(\bminus)_v \right] \\
	  &\quad+ \max_{v \in \VV} \Pr\left[ f_m^{+,v,s}(h_n^{+,v,r+s}(\xi_v^+))_v \neq f_m^{-,v,s}(h_n^{+,v,r+s}(\xi_v^+))_v \right].
	\end{align*}
This follows by two different types of monotonocity: if the configurations did not couple by time $n$, we may assume they take on their maximal difference. If they do, we may assume the  boundary conditions outside $\VV_{v,s}$ take on the worst possible state.

	For the first term, since $\VV_{u,r} \subset \VV_{v,r+s}$ for $u \in \VV_{v,s}$, we have for any $u \in \VV_{v,s}$ that
	\[ \Pr(E^c_{v,u}) \le \Pr\left[h_n^{+,v,r+s}(\bplus)_u \neq h_n^{-,v,r+s}(\bminus)_u\right] \le \Pr\left[h_n^{+,u,r}(\bplus)_u \neq h_n^{-,u,r}(\bminus)_u\right] \le \phi(n,r) ,\]
	so that
	\[ \Pr(E_v^c) \le \sum_{u \in \VV_{v,s}} \Pr(E^c_{v,u}) \le B(s) \phi(n,r) .\]
	For the second term, we note that $s \leq r$ implies that $\rho^{\bminus}_{\VV_{v,s}} \leD \rho^{\bminus}_{\VV_{v,r}} \leD \rho^{\bplus}_{\VV_{v,r}} \leD \rho^{\bplus}_{\VV_{v,s}}$. Since $\xi_v^+ \sim \rho^{\bplus}_{\VV_{v,r}}$, we see that
	\begin{align*}
	&f_m^{+,v,s}(h_n^{+,v,r+s}(\xi_v^+)) \le f_m^{+,v,s}(h_n^{+,v,s}(\xi_v^+)) \leD \rho^{\bplus}_{\VV_{v,s}} ,\\
	&f_m^{-,v,s}(h_n^{+,v,r+s}(\xi_v^+)) \ge f_m^{-,v,s}(h_n^{-,v,s}(\xi_v^+)) \geD \rho^{\bminus}_{\VV_{v,s}} .\end{align*}
	Thus, Lemma~\ref{lem:TV} gives that
	\[ \Pr\left[ f_m^{+,v,s}(h_n^{+,v,r+s}(\xi_v^+))_v \neq f_m^{-,v,s}(h_n^{+,v,r+s}(\xi_v^+))_v \right] \le |S| \cdot \big\|\rho^{\bplus}_{\VV_{v,s}}(\sigma_v \in \cdot) - \rho^{\bminus}_{\VV_{v,s}}(\sigma_v \in \cdot)\big\|_{TV} .\]
	Putting this together yields~\eqref{eq:recursion}.
\end{proof}

\begin{proof}[Proof of Lemma~\ref{lem:function-decay}]
	Denote $a_n := -\log \psi(n)$ and observe that the main assumption implies that
	\begin{equation}\label{eq:recursion-ineq}
	 a_{2n} \ge cs-\log (1+C) \qquad\text{for any }n \ge s \ge 1\text{ such that}\qquad cs+b(s) \le 2a_n .
	\end{equation}
	The restriction that $s \le n$ is a nuisance; to rid ourselves of it, we note that either $a_n \ge cn/2$ for infinitely many $n$, or, whenever $n$ is large, any solution to $cs+b(s) \le 2a_n$ satisfies $s \le n$ . In the former case, it is not difficult to check that $a_n = \Omega(n)$, so that $\psi(n)$ decays exponentially. We may therefore assume that this is not the case.
	
	We begin by showing that $\psi(n)$ decays faster than any stretched-exponential, that is, that $a_n$ grows faster than $n^\delta$ for any $0<\delta<1$.
	Since $b(s)$ is sub-linear by assumption, for any fixed $\epsilon>0$, we have $cs+b(s) \le 2x$ for all $s \le (2/c-\epsilon)x$ and large~$x$. Since $a_n \to \infty$ as $n \to \infty$, it follows from~\eqref{eq:recursion-ineq} that $a_{2n} \ge c(2/c-\epsilon)a_n - \log (1+C) \ge (2-\epsilon c-\epsilon)a_n$ for large $n$. We conclude that
	\[ \lim_{n \to \infty} a_{2^n} 2^{-\delta n} = \infty \qquad \text{for any }0<\delta<1. \]
	It is then straightforward to show that $a_n=n^{1-o(1)}$, establishing the first part of the lemma.
	
	Towards showing an exponential bound under the additional assumption on the growth rate of $b(n)$, let $\beta<c\log 2$ be such that $b(n) \le \frac{\beta n}{\log n}$. Let $\beta/c<\alpha<\log 2$. We claim that
\[ s \le \tfrac{2x}{c}(1-\tfrac{\alpha}{\log x}) \implies cs+b(s) \le 2x \qquad\text{for large }x .\]
	Indeed, since $s \mapsto cs+b(s)$ is increasing, this follows from
	\[ 2x(1-\tfrac{a}{\log x}) + \frac{\frac{2x\beta}{c}(1-\tfrac{\alpha}{\log x})}{\log \left[\tfrac{2x}{c}(1-\tfrac{\alpha}{\log x})\right]} = 2x(1-\tfrac{a}{\log x}) \cdot \left(1+ \frac{\beta}{c\log \left[\tfrac{2x}{c}(1-\tfrac{\alpha}{\log x})\right]}\right) \le 2x .\]
	Thus, by~\eqref{eq:recursion-ineq}, there exists $N$ such that
	\[ a_{2n} \ge 2a_n(1-\tfrac{\alpha}{\log a_n}) \qquad\text{for all }n \ge N .\]
	Let $0<\gamma<1-\frac{\alpha}{\log 2}$ and let $A>0$ be small enough so that $a_{2N} \ge AN / (\log N)^{1-\gamma}$.
	We prove by induction that
	\begin{equation*}
	a_n \ge \frac{An}{(\log n)^{1-\gamma}}  \qquad\qquad\text{for all }n \in \{N,2N,4N,8N,\dots \} .
	\end{equation*}
	Since $x \mapsto 2x(1-\frac{\alpha}{\log x})$ is increasing,
	\begin{align*}
	 a_{2n} \ge 2a_n\left(1-\frac{\alpha}{\log a_n}\right)
	  &\ge \frac{2An}{(\log n)^{1-\gamma}} \cdot \left(1- \frac{\alpha}{\log An- \log(\log n)^{1-\gamma}}\right) \\
	  &= \frac{2An}{(\log 2n)^{1-\gamma}} \cdot \frac{(\log 2n)^{1-\gamma}}{(\log n)^{1-\gamma}} \cdot \left(1- \frac{\alpha}{\log An- \log(\log n)^{1-\gamma}}\right) .
	\end{align*}
	The induction step now follows using that $(1-\gamma)\log 2 > \alpha$ and that
	\[ \frac{(\log 2n)^{1-\gamma}}{(\log n)^{1-\gamma}} = \left(1+\frac{\log 2}{\log n}\right)^{1-\gamma} \ge 1+\frac{(1-\gamma)\log 2}{\log n} .\]
	We conclude that
	\begin{equation}\label{eq:mixing-psi-decay1}
	\psi(n) \le \exp\left(-\frac{An}{(\log n)^{1-\gamma}}\right) \qquad\text{for all }n \in \{N,2N,4N,8N,\dots \} .
	\end{equation}
	Towards upgrading this bound to the desired exponential bound, define
	\[ \ell(n) := e^{\frac{n}{(\log n)^{1-\gamma/2}}} \psi(n) + e^{-cn/4 + \sqrt{n}} .\]
	Using the recursion assumption with $s=n$, we see that for large $n$,
	\begin{align*}
	\ell(2n)
	&= e^{\frac{2n}{(\log 2n)^{1-\gamma/2}}} \psi(2n) + e^{-cn/2 + \sqrt{2n}} \\
	&\le e^{\frac{2n}{(\log 2n)^{1-\gamma/2}}+\frac{cn}{\log n}} \psi(n)^2 + e^{\frac{2n}{(\log 2n)^{1-\gamma/2}}-cn} + e^{-cn/2 + \sqrt{2n}} \\
	&\le e^{\frac{2n}{(\log n)^{1-\gamma/2}}} \psi(n)^2 + e^{-cn/2 + \sqrt{2n} + 1} \\
	&\le \ell(n)^2 .
	\end{align*}
	Since~\eqref{eq:mixing-psi-decay1} implies that $\liminf_{n \to \infty} \ell(n)=0$, there exists $M \ge 1$ such that $\ell(M) \le 1/e$ and $\ell(2n) \le \ell(n)^2$ for all $n \ge M$. Then $\ell(2^nM) \le e^{-2^n}$ for all $n \ge 0$. Since $\max_{n \le m \le 2n} \ell(m) \le e^{n/(\log n)^{1-\gamma/2}} \ell(n)$, it easily follows that $\ell(n)$ decays exponentially fast, and we conclude that $\psi(n)$ also decays exponentially fast.
\end{proof}

\section{Space-time finitary codings}
\label{sec:space-time}
The goal of this section is to prove Theorem \ref{thm:fv}. As such, the graph $G$ will be an infinite quasi-transitive graph satisfying~\eqref{eq:sphere-condition}, $S$ will be finite (and identified with $\{0,1,\dots,|S|-1\})$, and $\rho$ will be a monotone $\Gamma$-invariant irreducible marginally finite upwards-downwards specification such that $\mu^+ = \mu^-$.

We will reuse the dynamics of Section \ref{sec:MainCoupling}, but will now demand that the spaces $\cA$ and $\cB$ are both finite. Before we construct suitable versions of $(\cA,\pi)$, $F$, $(\cB,\theta)$ and $\mathcal{O}$, let us outline the properties required to obtain a space-time finitary coding.

For any $v \in \VV$, define 
\[
T_v := \min\big\{ n : f_n^{+,v,n}(\bplus)_v = f_n^{-,v,n}(\bminus)_v \big\}.
\]
The assumption $\mu^+ = \mu^-$ implies that $\sigma^+ = \sigma^-$ (as was shown in the proof of Theorem \ref{thm:coding-for-specs}, item 2). Thus, from Corollary \ref{cor:ASconvergence}, we conclude that $T_v$ is almost surely finite, and, in particular,  
\[ \sigma^+_v = \sigma^-_v = f^{+,v,T_v}_{T_v}(\bplus) = f^{-,v,T_v}_{T_v}(\bminus) , \qquad v \in \Z^d. \]
With a finite-valued construction of $(\cA, \pi)$ and $(\cB,\theta)$, this does not allow us to conclude that the coding $\varphi$ is finitary (let alone space-time finitary). Indeed, $T_v$ does not bound the  coding radius as it is not necessarily a stopping time with respect to the filtration $(Y_{\VV_{v,r}})$. This is because the ordering $\preceq_{B_n}$ restricted to $\VV_{v,r}$ may depend on $\{B_{u,n}\}_{ u \not \in \VV_{v,r}}$.

To deal with this issue, for $\eta \in \cB^\VV$, we define 
\[
R_{u,v}(\eta) := \min \big\{r \ge 0 : \1_{\{u \preceq_{\eta} v\}}=\1_{\{u \preceq_{\eta'} v\}} \text{ for any } \eta' \text{ satisfying } \eta'_{\VV_{u,r}\cup\VV_{v,r}} = \eta_{\VV_{u,r}\cup\VV_{v,r}} \big\},
\]
where we again set the variable to $\infty$ if the set is empty. In words $R_{u,v}(\eta)$ is the minimal radius around $u$ and $v$ needed to determine the relative $\preceq_\eta$-order between $u$ and $v$. We now set 
\[
T^*_v := 2\min\big\{ n : f_n^{+,v,n}(\bplus)_v = f_n^{-,v,n}(\bminus)_v\text{ and }R_{u,w}(B_i) \le n\text{ for all }u,w \in \VV_{v,n}\text{ and }0 \le i \le n \big\}.
\]
The factor of 2 is introduced to accommodate the fact that $\VV_{u,n} \subset \VV_{v,2n}$ for $u \in \VV_{v,n}$.
Then $T^*_v$ is a bound on the coding radius at $v$. More importantly, it is a bound on the space-time coding radius, as one can easily see. Thus, the first part of Theorem~\ref{thm:fv} will follow once we can construct finite probability spaces $(\cA,\pi)$ and $(\cB,\theta)$ and their associated functions that imply $T^*_v$ is almost surely finite. The second part will require quantitative bounds on the tails of $T^*_v$, which will require the use of the mixing time results of the previous section.

\subsection{Constructing finite probability spaces}
We begin by choosing $(\cA, \pi)$ and $F$.
We set
\[ \cA := \big\{ \rho^{\omega}_v [ \sigma_v \leq s] \big\}_{\omega \in \Omega^+ \cup \Omega^-,~v\in \VV,~ s \in S}. \]
Since $S$ is finite and $\rho$ is marginally finite, we immediately see that $\cA$ is finite as well.
Thus, we may order the finite number of elements of $\cA$ in increasing order, $0\le a_1<\dots<a_m=1$, where $m := |\cA|$. Letting $a_0 := 0$, we define $\pi$ by
\[ \pi(\{a_i\}) :=  a_i - a_{i-1}, \qquad 1 \le i \le m . \]
We then define $F$ by
\[ F(\omega,v,a_i) := \min \{ s \in S :\rho^{\omega}_v [ \sigma_v \leq s] \ge a_i \} , \qquad \omega \in \Omega^+ \cup \Omega^-,~ 1 \le i \le m .\]
It is straightforward to check that \eqref{eq:FPlusDef}, \eqref{eq:FPlusMon} and~\eqref{eq:FPlusStabInvariant} hold, where we recall that, in all those equations, we allow $\omega \in \Omega^+ \cup \Omega^-$.

We now turn to choosing $(\cB,\theta)$ and $\cO$.
We set $\cB := \{1,\dots,D\}$ for some integer $D \ge 2$ and we set $\theta$ to be the uniform measure on $\cB$.
Given $\eta \in \cB^\VV$ and $v \in \VV$, define $Z_v(\eta)=(Z_{v,n}(\eta))_{n \ge 0} \in \N^\N$ by
\[
Z_{v,n}(\eta) := \sum_{u \in \VV_{v,n} \setminus \VV_{v,{n-1}}} \eta_u, 
\]
where it is understood that $Z_{v,0}(\eta) := \eta_v$.
We now define
\[
\cO(\eta,u,v) := \mathbf{1}_{Z_u(\eta) \leq Z_v(\eta)},
\]
where $\leq$ is used to indicate the lexicographical order on $\mathbb{N}^{\mathbb{N}}$. This creates a preorder $\preceq_\eta$ on $\VV$. The following lemma shows that $\cO$ is $\theta$-compatible, i.e., that $\preceq_\eta$ is almost surely a total ordering, when $(\eta_v)_{v \in \VV}$ are \iid\ samples of $\theta$.

\begin{lemma}\label{lem:finiteordering}
Let $G$ be an infinite quasi-transitive graph satisfying~\eqref{eq:sphere-condition}.
Then $\mathcal{O}$ is $\theta$-compatible and, letting $\eta=(\eta_v)_{v \in \VV}$ be \iid\ random variables sampled from $\theta$,
\[
\Pr(R_{u,v}(\eta) > r) \le D^{-r} \qquad\text{for any distinct } u,v \in \VV\text{ and }r \ge 0 .
\]
\end{lemma}
\begin{proof}
	Fix $u,v \in \VV$ distinct. Consider the event
	\[ A_n := \bigcap_{i=1}^n \{ Z_{u,i}(\eta) = Z_{v,i}(\eta) \} .\]
	Observe that $\mathcal{O}$ is $\theta$-compatible if and only if $\Pr(A_n) \to 0$ as $n \to \infty$.
	Observe also that $R_{u,v}(\eta)>n$ implies the occurrence $A_n$.
	Thus, the lemma will follow once we show that $\Pr(A_n \mid A_{n-1}) \le \frac1D$ for all $n \ge 1$.
	By~\eqref{eq:sphere-condition}, there exists some $w_n \in (\VV_{u,n} \setminus \VV_{u,n-1}) \setminus \VV_{v,n}$. Then
	\[ \Pr\big(A_n \mid \eta_{\VV \setminus \{w_n\}}\big) \le \max_{k \in \Z} \Pr(\eta_{w_n}=k) \le \tfrac1D .\]
	Since $A_{n-1}$ is measurable with respect to $\eta_{\VV \setminus \{w_n\}}$, it follows that $\Pr(A_n \mid A_{n-1}) \le \frac1D$.
\end{proof}

\subsection{Proof of Theorem~\ref{thm:fv}}
The first item of the theorem will follow once we show that $T^*_v$ is almost surely finite.
By monotonicity, $f^{+,v,n}_n(\bplus)_v = f^{-,v,n}_n(\bminus)_v $ for all $n \ge T_v$. Thus, since $T_v$ is almost surely finite, it suffices to show that $\Pr(E_n) \to 0$ as $n \to \infty$, where $E_n$ is the event that there exists $0 \le i \le n$ and a pair of vertices $u,w \in \VV_{v,n}$ for which $R_{u,w}(B_i) \ge n$.
Indeed, taking $D$ to be larger than $3\Delta^2$, where $\Delta$ is the degree of $G$, the union bound and Lemma~\ref{lem:finiteordering} allow us to conclude that
\[
\Pr[E_n] \le (n+1) B(n)^2 D^{-n} \le 10n \Delta^{2n} D^{-n} \le 2^{-n} \qquad \text{for any sufficiently large }n.
\]
This completes the proof of the first item of the theorem.

We now turn to the second item of the theorem.
Since, on the complement of $E_n$, $T_v \le n$ implies that $T^*_v \le 2n$, we see that
\begin{align*}
\Pr[T^*_v > 2n]
 \le \Pr[E_n] + \Pr[\big\{T^*_v > 2n \big\} \cap E_n^c]
 \le \Pr[E_n] + \Pr[T_v > n] .
\end{align*}
Thus, for large $n$,
\[ \Pr[T^*_v > 2n] \le 2^{-n} + |S| \cdot \big\|f_n^{+,v,n}(\bplus)_v - f_n^{-,v,n}(\bminus)_v\big\|_{TV} ,\]
and the theorem follows from Theorem \ref{thm:mixing}.

\section{Codings from finite-valued \iid\ processes} \label{sec:FiniteValued}

In this section, we prove Corollary~\ref{cor:fv-coding}. To this end, we require a result from~\cite{spinka2018finitaryising} which allows to convert space-time finitary codings with exponential tails to \fvffiid\ with stretched-exponential tails. 

Suppose that $Y=(Y_{v,i})_{v \in \Z^d, i \ge 0}$ are \iid\ random variables taking values in a finite set $\tilde{T}$.
Let $\zero$ be the origin of $\Z^d$, and $F=(F_n)_{n \ge 0}$ be a strictly increasing sequence of subsets of $\Z^d \times \N$ with $F_0:=\{(\zero,0)\}$, and consider the associated $\sigma$-algebras $\{\cF^n_v\}_{v \in \Z^d, n \ge 0}$ defined by
\begin{equation}\label{eq:F_n}
	\cF^n_v := \sigma\big(\{Y_{v+u,i}\}_{(u,i) \in F_n}\big) .
\end{equation}
An $\N$-valued random field $\tau=(\tau_v)_{v \in \Z^d}$ is said to be a \emph{$F$-stopping-process} for $Y$ if, for every $v$, $\tau_v$ is an almost surely finite stopping time with respect to the filtration $(\cF^n_v)_{n \ge 0}$. When we say that such a stopping-process is stationary, we shall mean that the same stopping rule is used at every vertex (rather than just meaning that its law is translation-invariant).
Given a $F$-stopping-process, we denote by $Y^\tau$ the random field
\[ Y^\tau := \big((Y_{v+u,i})_{(u,i) \in F_{\tau_v}}\big)_{v \in \Z^d} .\]
Note that $(Y^\tau)_v$ takes values in the finite-configuration space $\bigcup_{n \ge 0} \tilde{T}^{F_n}$.
We say that $F$ is \emph{linear} if
\begin{equation}\label{eq:linear-stopping-process}
\Delta_n := \max \big\{ \max\{|u|,i\} : (u,i) \in F_n\big\} \le \Delta n \qquad\text{for some $\Delta \ge 1$ and all $n \ge 0$.}
\end{equation}

\begin{proposition}[{\cite[Proposition~10]{spinka2018finitaryising}}]\label{prop:coding}
Let $Y=(Y_{v,i})_{v \in \Z^d, i \ge 0}$ be a finite-valued \iid\ process, let $F$ be linear and let~$\tau$ be a stationary $F$-stopping-process for $Y$. Suppose that $\tau_v$ has exponential tails and $\E |F_{\tau_v}| < M$ for some integer $M$. Then $Y^\tau$ is a translation-equivariant finitary factor of $((Y_{v,i})_{0 \le i < M})_{v \in \Z^d}$ with stretched-exponential tails.
\end{proposition}

\begin{proof}[Proof of Corollary~\ref{cor:fv-coding}]
	In the proof, all factors are translation-equivariant, i.e., $\Gamma$-factors where $\Gamma$ is the group of translations.

	Suppose first that $G=\Z^d$. In this case, $\mu^+$ is a measure on $S^{\Z^d}$.
    By Theorem~\ref{thm:fv}, there exists a space-time finitary coding $\varphi$ from an \iid\ process $Y=(Y_{v,i})_{v \in \Z^d,i \ge 0}$ to $\mu^+$ whose space-time coding radius $R^*$ has exponential tails. Let $R^*_v$ denote the space-time coding radius of the vertex $v \in \VV$.
	Towards applying Proposition~\ref{prop:coding}, define $F_n := \{ (u,i) : |u| \le n,~0 \le i \le n \}$ and note that the random field $\tau=(R^*_v)_{v \in \Z^d}$ is a stationary $F$-stopping-process for $Y$. Since $\tau_v$ has exponential tails, it follows that $\E |F_{\tau_v}|<\infty$. Note that, by definition of the process $Y^\tau$, there exists a deterministic function $\psi$ such that $\varphi(Y)_v = \psi((Y^\tau)_v)$ for all $v$ ($\psi$ may, in some sense, be thought of as $\varphi(\cdot)_\zero$). In particular, $\mu^+$ is a finitary factor of $Y^\tau$ with coding radius~0. It thus suffices to show that $Y^\tau$ is \fvffiid\ with stretched-exponential tails.	
	Indeed, letting $M$ be any integer larger than $\E |F_{\tau_v}|$, Proposition~\ref{prop:coding} yields that $Y^\tau$ is a finitary factor of $((Y_{v,i})_{0 \le i < M})_{v \in \Z^d}$ with stretched-exponential tails. Since the latter process is a finite-valued \iid\ process, this yields the required coding for $\mu^+$, completing the proof in the case of $G=\Z^d$.
	
	Suppose now that $G$ is the line graph of $\Z^d$. In this case, $\mu^+$ is a measure on $S^{E(\Z^d)}$ and the coding $\varphi$ obtained from Theorem~\ref{thm:fv} is from an \iid\ process $Y'=(Y'_{e,i})_{e \in E(\Z^d),i \ge 0}$.
Converting the problem to one on the vertices of $\Z^d$ is simple: Let $X'$ be sampled from $\mu^+$, and define $X=(X_v)_{v\in \Z^d}$ and $Y=(Y_{v,i})_{v \in \Z^d, i \ge 0}$ by $X_v = (X'_{e_1},\dots,X'_{e_d})$ and $Y_{v,i} = (Y'_{e_1,i},\dots,Y'_{e_d,i})$, where $e_j = \{v,v+e_j\}$ for $1\le j \le d$. Clearly, $X$ is a space-time finitary factor of $Y$ with space-time coding radius having exponential tails. The argument above now shows that, for some $M$, $X$ is a a finitary factor $((Y_{v,i})_{0 \le i < M})_{v \in \Z^d}$ with stretched-exponential tails.
Going back to the prime processes, we obtain that $X'$ is a finitary factor of $((Y'_{e,i})_{0 \le i < M})_{e \in E(\Z^d)}$ with stretched-exponential tails, as required. We note that we have relied very much on the fact that $\Gamma$ is the translation group (and not a larger group).
\end{proof}

\bibliographystyle{amsplain}
\bibliography{library}

\end{document}